\documentclass[12pt,a4paper,oneside]{amsart}
\usepackage[utf8]{inputenc}
\usepackage{roboto}
\usepackage{anysize}
\marginsize{2cm}{2cm}{1.5cm}{1.5cm}
\usepackage{amsfonts, amsmath, amssymb, amsthm, amscd}
\usepackage{upgreek}
\usepackage[english]{babel}
\usepackage{enumitem}
\usepackage{hyperref}
\hypersetup{
	colorlinks   = true, 
	urlcolor     = blue, 
	linkcolor    = blue, 
	citecolor    = red   
}

\numberwithin{equation}{section}
\newtheorem{problem}{Problem}
\newtheorem{thm}{Theorem}[section]
\newtheorem{prp}[thm]{Proposition}
\newtheorem{lem}[thm]{Lemma}
\newtheorem{cor}[thm]{Corollary}

\newtheorem{conj}[thm]{Conjecture}
\theoremstyle{definition}
\newtheorem{dfn}[thm]{Definition}
\newtheorem{rem}[thm]{Remark}

\newcommand{\Href}[2]{\hyperref[#2]{#1~\ref{#2}}}
\newcommand{\Hrefs}[3]{\hyperref[#2]{#1~\ref{#2}} and \hyperref[#3]{\ref{#3}}}
\newcommand{\Hreftitle}[2]{\texorpdfstring{
    \hspace{-4pt}\Href{#1}{#2}}{#1~\ref{#2}}}
\newcommand{\Hrefstitle}[3]{\texorpdfstring{%
    \Hrefs{#1}{#2}{#3}}{#1~\ref{#2} and \ref{#3}}%
}
\def\phi{\varphi}
\def\epsilon{\varepsilon}
\def\alpha{\upalpha}
\renewcommand{\Re}{\mathbb{R}}
\newcommand{\Red}{\Re^d}

\newcommand{\ReN}{\Re^n}

\newcommand{\ball}[1]{\mathbf{B}^{#1}}%
\newcommand{\set}[1]{\left\{#1\right\}}
\newcommand{\braces}[1]{\left\{#1\right\}}
\newcommand{\norm}[1]{\left\|#1\right\|}

\newcommand{\conv}[1]{\operatorname{conv} #1 }
\newcommand{\posname}{\operatorname{pos}}
\newcommand{\pos}[1]{\posname\left(#1\right)}

\newcommand{\bd}[1]{\operatorname{bd} #1}
\newcommand{\cl}[1]{\operatorname{cl}\left(#1\right)}
\newcommand{\inter}[1]{\mathrm{int} #1}
\newcommand{\iprod}[2]{\left\langle#1,#2\right\rangle}
\newcommand{\iprodsp}[3][\rotpos]{\left\langle#2,#3\right\rangle_{#1}}
\newcommand{\vol}[2][d]{\operatorname{vol}\nolimits_{#1}\left(#2\right)}%
\newcommand{\volnf}[2][d]{\operatorname{vol}\nolimits_{#1} #2}%
\newcommand{\tr}[1]{\operatorname{tr}\left(#1\right)}
\newcommand{\trace}[1]{\operatorname{tr} #1 }
\newcommand{\rank}[1]{\operatorname{rank}\left(#1\right)}
\newcommand{\inrad}[1]{\operatorname{inrad} #1}

\newcommand{\st}{:\;}
\newcommand{\di}{\,\mathrm{d}}
\newcommand{\pol}{^{\circ}}

\newcommand{\enorm}[1]{\left|#1\right|}
\newcommand{\lin}[1]{\operatorname{Lin}\left(#1\right)}

\newcommand{\crosp}{\Diamond}%

\newcommand{\noshow}[1]{}

\newcommand{\contact}{\mathrm{Cont}_{\Red}}

\newcommand{\sqmatrix}[1]{\mathcal{M}^{#1}}
\newcommand{\Symmatrix}[1]{\mathcal{M}_{\mathrm{Sym}}^{#1}}
\newcommand{\ucontopcr}{\mathrm{Cont}_{\mathcal{R}}}

\newcommand{\uconttensorcr}{\mathcal{C}_{\mathcal{R}}}

\newcommand{\rotpos}{\mathcal{R}}

\newcommand{\gammaderivatives}{\mathcal{D}}
\newcommand{\transpose}[1]{{#1}^{\top}}%
\newcommand{\inv}[1]{{#1}^{-1}}%
\newcommand{\ort}[1]{{#1}^{\bot}}%
\newcommand{\conerotpos}{\mathcal{V}_{\rotpos}}

\providecommand{\parenth}[1]{\left(#1\right)}
\providecommand{\braces}[1]{\left\{#1\right\}}

\newcommand{\id}{\mathrm{Id}}
\newcommand{\origin}{\mathbf{0}}
\newcommand{\originmsp}[1]{\mathbf{0}_{#1}}
\newcommand{\aposition}[1][F]{\mathcal{A}_{#1}}%
\newcommand{\apositions}[2][F]{\aposition[#1] \! \parenth{#2}}%
\newcommand{\rposition}[1][F]{\mathcal{R}_{#1}}%
\newcommand{\rpositions}[2][F]{\rposition[#1] \! \parenth{#2}}%
\def\polar{\circ}
\newcommand{\polarset}[1]{{#1}^{\polar}}%
\newcommand{\indicator}[1]{\mathbf{1}_{#1}}%
\newcommand{\ncone}[2]{N_{#1}\!\parenth{ #2}}%
\newcommand{\poscone}[1]{\mathrm{Pos}\ \! {#1}}%
\newcommand{\tancone}[2]{\mathrm{T}_{#1} \! \parenth{#2}}%
\newcommand{\cube}{\Box}%
\newcommand{\diffzero}{\left.\frac{\di}{\di t}\right|_{t=0^+}}

\title{John Ellipsoids of Revolution}
\author{Grigory Ivanov}
\address{G.~Ivanov: Pontif\'icia Universidade Cat\'olica do Rio de Janeiro\\
Departamento de Matem\'atica\\
Rua Marqu\^es de S\~ao Vicente, 225\\
Edif\'{\i}cio Cardeal Leme, sala 862\\
22451-900 G\'avea, Rio de Janeiro, Brazil}
\email{grimivanov@gmail.com}

\author{Zsolt L\'angi}
\address{Z.~L\'angi: Bolyai Institute, University of Szeged and HUN-REN Alfr\'ed R\'enyi Institute of Mathematics, Hungary}
\email{zlangi@server.math.u-szeged.hu}

\author{M\'arton Nasz\'odi}
\address{M.~Nasz\'odi: HUN-REN Alfr\'ed R\'enyi Institute of Mathematics and Lor\'and E\"otv\"os University, Budapest, Hungary}
\email{marton.naszodi@renyi.hu}

\author{\'Ad\'am Sagmeister}
\address{\'A.~Sagmeister: Bolyai Institute, University of Szeged, Aradi vértanúk tere 1, H-6720 Szeged,
Hungary}
\email{sagmeister.adam@gmail.com}


\begin{document}
\begin{abstract}
Finding a \emph{largest Euclidean ball} in a given convex body $K \subset \mathbb{R}^d$ and finding a \emph{largest volume ellipsoid} in $K$ are two problems of fundamentally different nature. The first is a purely Euclidean problem, where we consider scaled copies of the origin-centered closed unit ball, whereas in the second problem, we search among all affine copies of the unit ball.

In this paper, we interpolate between these two classical problems by considering ellipsoids of revolution. More generally, we study pairs of convex bodies $K$ and $L$, and seek a largest-volume affine image of $K$ contained within $L$, subject to certain restrictions on the allowed affine transformations. We derive first-order necessary conditions for optimality, generalizing known conditions from the unrestricted affine setting.

Using these conditions, we show that an extremal ellipsoid of revolution exhibits properties analogous to those of either the largest-volume ellipsoid or the largest Euclidean ball, depending on whether the ellipsoid is considered along its axis of revolution or along the orthogonal complement of that axis.

\end{abstract}

\vspace{-1cm}
\maketitle

\section{Introduction}

Finding the \emph{largest Euclidean ball} contained in a convex body $K \subset \Red$ and finding the \emph{largest volume ellipsoid} contained in $K$ are problems of fundamentally different nature. The former is a Euclidean problem involving scaled copies of the origin-centered unit ball $\ball{d}$, while the latter involves all affine images of $\ball{d}$. In this paper, we study a natural interpolation between these two settings by considering \emph{ellipsoids of revolution}. More generally, we investigate pairs of convex bodies $K$ and $L$ and aim to determine the largest volume affine image of $K$ contained in $L$, subject to specific restrictions on the class of admissible affine transformations. Our starting point is John’s classical theorem \cite{Jo48, Jo14} characterizing the maximal volume ellipsoid in a convex body and its generalizations to pairs of convex bodies \cite{L79, TJ89, BR02, GLMP04}. Our goal is to initiate the development of a similar theory for ellipsoids of revolution and for constrained affine images in this more general setting.

\subsection{Ellipsoids of revolution} In order to introduce ellipsoids of revolution, we start with defining a special class of linear operators. 

\begin{dfn}
For a linear subspace $F$ in $\Red$, we call an invertible linear operator $A$ on $\Red$ an \emph{$F$-operator}, if both $F$ and $\ort{F}$ are invariant subspaces of $A$, and the restriction $A|_{\ort{F}}$ of $A$ on the orthogonal complement $\ort{F}$ of $F$ is a (nonzero) multiple of the identity operator $\id_{\ort{F}}$ on $\ort{F}$.
\end{dfn}

Note that if we take an orthonormal basis of $F$, and extend it to an orthonormal basis of $\Red$, then the matrix of an $F$ operator takes the form 
\[
M = \begin{pmatrix} {M_1}& \mathbf{0}_{s \times (d-s)}\\ \mathbf{0}_{ (d-s) \times s} & \mu \id_{d-s}\end{pmatrix},
\]
where $M_1 \in \Re^{s\times s},$ $\mu \in \Re\setminus\{0\},$ 
$\mathbf{0}_{s \times (d-s)}$  and $\mathbf{0}_{ (d-s) \times s}$ are zero matrices of the appropriate sizes, and $\id_{d-s}$ denotes the identity matrix of size $d-s$.

\begin{dfn}
    Let $F$ be a linear subspace in $\Red$ and $S$ be any subset of $\Red$.
We will refer to the family of affine transformations
\[
 \aposition=\braces{x\mapsto Ax+z\st A \text{ is an } F\text{-operator, } z\in\Red}
\]
as the \emph{affine transformations with axis $F$}, and the family of sets 
\[
 \apositions{S}=\braces{AS+z\st A \text{ is an } F\text{-operator, } z\in\Red}
\]
as the \emph{positions of $S$ with axis $F$}.

In particular, we call members of $\apositions{\ball{d}}$ \emph{ellipsoids of revolution with axis $F$.} 
\end{dfn}

The naming is justified, since if $E\in\apositions{\ball{d}}$, then $E$ is an ellipsoid all of whose non-degenerate sections by translates of $\ort{F}$ are $(d-\dim{F})$-dimensional balls.
Our first problem is the following. 

\newcommand{\probellfixed}{John's problem for ellipsoids with a fixed axis of revolution}
\begin{problem}[\probellfixed]\label{prob:ellfixed}
For a given convex body (that is, a closed, convex set with non-empty interior) $K$ in $\Red$ and a linear subspace $F$ in $\Red$, find a largest volume ellipsoid of revolution with axis $F$ contained in $K$.
\end{problem}

Clearly, when $F=\{\origin\}$ is the 0-dimensional subspace, then this problem is equivalent to finding a largest volume Euclidean ball in $K$. On the other hand, when $F=\Red$, then it is exactly John's problem of finding the largest volume ellipsoid contained in $K$.

In \probellfixed, the subspace $F$ is fixed. It is natural to allow for the variation of $F$, which leads us to our second problem.

\newcommand{\probellrotation}{John's problem for ellipsoids with any axis of re\-vo\-lu\-tion}
\begin{problem}[\probellrotation]\label{prob:ellrotation}
Let $K$ be a convex body in $\Red$ and $s\in\{0,1,\dots,d\}$. Find a largest volume ellipsoid of revolution in $K$ with any $s$-dimensional axis.
\end{problem}

\subsection{A more general problem}

The problems above are motivated by John's fundamental theorem \cite{Jo48, Jo14}, where a necessary condition is given for the ball $\ball{d}$ to be the maximal volume ellipsoid contained in a convex body $K$. The sufficiency of the same condition was noted by K. Ball \cite{ball1997elementary}.

A natural generalization of this question was studied by V.~Milman \cite[Theorem~14.5]{TJ89}, Lewis \cite{L79}, Giannopoulos, Perissinaki and Tsolomitis \cite{GPT01}, Bastero and Romance \cite{BR02}, and Gordon, Litvak, Meyer and Pajor \cite{GLMP04}: given convex bodies $K$ and $L$ in $\Red$, find a maximal volume affine image of $K$ that is contained in $L$. In this work, we study the same problem under additional restrictions on the affine transformation.

\newcommand{\probgenfixed}{John's general problem with fixed axis}
\begin{problem}[\probgenfixed]\label{prob:genfixed}
For given convex bodies $K$ and $L$, and a linear subspace $F$ in $\Red$, find a largest volume member of $\apositions{K}$ contained in $L$.
\end{problem}

Clearly, when $F=\{\origin\}$ is the 0-dimensional subspace, then this problem is equivalent to finding a largest \emph{homothet} of $K$ in $L$. On the other hand, when $F=\Red$, then it is exactly the problem described in the paragraph preceding Problem~\ref{prob:genfixed}.

\begin{dfn}
Let $F$ be a linear subspace in $\Red$ and $S$ be any subset of $\Red$.
We will refer to the family of affine transformations
\[
 \rposition=\braces{x\mapsto RAx+z\st R\in O(d), A \text{ is an } F\text{-operator, } z\in\Red},
\]
where $O(d)$ denotes the orthogonal group on $\Red$,
as the \emph{rotated affine transformations with axis $F$}, and the family of sets
\[
 \rpositions{S}=\braces{RAS+z\st R\in O(d), A \text{ is an } F\text{-operator, } z\in\Red},
\]
as the \emph{rotated positions of $S$ with axis $F$}.
\end{dfn}

\newcommand{\probgenrotation}{John's general problem with rotated axis}
\begin{problem}[\probgenrotation]\label{prob:genrotation}
Let $K$ and $L$ be convex bodies in $\Red,$ and  fix a linear subspace $F$ in $\Red$. Find a largest volume rotated position of $K$ with axis $F$ that is contained in $L$.
\end{problem}

Note a subtle but very important difference between this general case and ellipsoids. To define ellipsoids of revolution with an $s$-dimensional axis, we only need to specify the dimension $s$ of the axis, but not the axis $F$. In the general case introduced here, we fix $F$, apply an $F$-operator $A$ and \emph{then} allow a rotation $R$. The advantage of this is that in this way, we can specify what `part' of $K$ we allow to be changed by an affine transformation (the $F$-component), and what part (the $\ort{F}$ component) is rigid, that is, it only allows similarities. A simple illustration is the task of finding the largest area isosceles triangle contained in a planar set $L$. In our framework, we simply take $K$ to be a regular triangle, and $F$ to be one of its altitude lines.  Then $\rpositions{K}$ is exactly the set of isosceles triangles.

When $F=\{\origin\}$, then \Href{Problem}{prob:genrotation} is the problem discussed at the beginning of this subsection considered by Lewis, Milman and others. On the other hand, when $F=\Red$, then we are looking for the largest \emph{similar} copy of $K$ in $L$.

\subsection{Results}

Our first goal is to find conditions of optimality in Problems~\ref{prob:ellfixed}, \ref{prob:ellrotation}, \ref{prob:genfixed} and \ref{prob:genrotation} similar to John's condition in the case of ellipsoids and the condition given by Gordon, Litvak, Meyer and Pajor \cite{GLMP04} in the general case. We start with \Hrefs{Problems}{prob:genfixed}{prob:genrotation}, as the optimality conditions in the case of ellipsoids of revolution will be obtained as corollaries to these general results.

Let $K$ and $L$ be convex bodies in $\Red$ with $K\subseteq L$. Following \cite{GLMP04}, we call $(v,p)\in\Red\times\Red$ a 
\emph{contact pair} of $K$ and $L$ if $v \in \bd{K}\cap\bd{L},$ and $p \in \bd{\polarset{K}} \cap \bd{\polarset{L}} $  normalized by the condition $\iprod{p}{v}=1$. In other words, $p$ is an outer normal vector of a common supporting hyperplane of $K$ and $L$ at $v$, with the appropriate normalization. 
We use $P_F$ to denote the \emph{orthogonal projection} onto $F$. We define the \emph{diadic product} of two vectors $v$ and $p$ by $(p\otimes v)x=\iprod{v}{x}p$, a linear operator on $\Red$. Note that the matrix of $p\otimes v$ in the standard basis is $p\transpose{v}$, where $v$ and $p$ are column vectors. For a positive integer $m$, we use the notation $[m]$ for the set $[m]=\{1,2,\dots,m\}$.

Our necessary condition of optimality for \Href{Problem}{prob:genfixed} 
is the following.

\begin{thm}[Condition for \probgenfixed{}]\label{thm:genfixed}
Let $K, L \subset \Red$ be convex bodies containing the origin in their interiors and such that $K \subseteq L$, and let $F$ be a linear subspace in $\Red$.
Assume that $K$ is a solution to \probgenfixed{} $F$ for $K$ and $L$, that is, \Href{Problem}{prob:genfixed}. Then there are 
\begin{enumerate}
    \item contact pairs $(v_1, p_1),\dots, (v_m, p_m)$ of $L$ and $K$,
    \item positive weights $\alpha_1,\dots,\alpha_m$,
\end{enumerate}
such that
  \begin{equation}\label{eq:answerGLMPfixedsubspace_decomposition}
    P_F \parenth{\sum\limits_{i \in [m]} \alpha_i p_i\otimes v_i} P_F = P_F;
  \end{equation}
  \begin{equation}\label{eq:answerGLMPfixedsubspace_zerosum}
    \sum\limits_{i \in [m]} \alpha_i p_i =\origin ;
  \end{equation}
  \noshow{
   \begin{equation}\label{eq:answerGLMPfixedsubspace_zerosumv}
     \sum\limits_{i \in [m]} \alpha_i P_F v_i =\origin ;
   \end{equation}
   }
    \begin{equation}\label{eq:answerGLMPfixedsubspace_trace}
  \tr{\sum\limits_{i \in [m]} \alpha_i p_i\otimes v_i} = \sum\limits_{i \in [m]} \alpha_i = d.
  \end{equation}
\end{thm}

As a corollary, we will obtain the following.
\begin{thm}[Containment for \probgenfixed{}]\label{thm:GLMP_inclusion_in_F}
In addition to the assumptions of \Href{Theorem}{thm:genfixed}, suppose that 
\( P_F K \subset L \cap F \). Then there exists a translation vector 
\( z\) in the relative interior of \( P_F K \) such that
\[
P_F (K - z) \subset \parenth{L - z} \cap F \subset -d \cdot P_F (K - z).
\]
In particular, such a translation vector \( z \) exists whenever \( K \) is symmetric with respect to \( F \).
\end{thm}

In the case $F=\Red$, \Href{Theorem}{thm:genfixed} and 
\Href{Theorem}{thm:GLMP_inclusion_in_F} were obtained in \cite{GLMP04}.

Our necessary condition of optimality for \Href{Problem}{prob:genrotation} is the following.

\begin{thm}[Condition for \probgenrotation{}]\label{thm:genrotation}
With the notation of \Href{Theorem}{thm:genfixed}, assume that $K$ contains the origin in its interior and is a solution to \probgenrotation{} for $K$, $L$ and the subspace $F$. Then  there are
contact pairs of $L$ and $K$, and weights as in \Href{Theorem}{thm:genfixed} such that 
in addition to identities  \eqref{eq:answerGLMPfixedsubspace_decomposition} -- \eqref{eq:answerGLMPfixedsubspace_trace}, we also have 
\begin{equation}
    \label{eq:answerGLMProtation}
    \sum\limits_{i \in [m]} \alpha_i p_i\otimes v_i = 
     \sum\limits_{i \in [m]} \alpha_i v_i\otimes p_i.
\end{equation}
\end{thm}

We will state our analogous results for ellipsoids of revolution, that is for \Hrefs{Problems}{prob:ellfixed}{prob:ellrotation} in \Href{Section}{sec:ellipsoids}. Here we state other geometric properties of maximal volume ellipsoids of revolution.

We define the \emph{inradius} $\inrad A$ of a compact, convex set $A$ as the radius of a largest radius ball inside the affine hull of $A$ contained in $A$.

\begin{thm}[Properties of solutions of \probellfixed]\label{thm:ellfixedproperties}
    Let an origin-centered ellipsoid $E$ be a solution to \Href{Problem}{prob:ellfixed} for a convex body $K$ and a linear subspace $F$ in $\Red$ of dimension $s\in\{0,1,\dots,d-1\}$. 
    Then 
\begin{enumerate}[label=({\arabic*})]
    \item\label{item:inradius}
    the inradius of $K \cap \ort{F}$  satisfies
\[
 \inrad \parenth{K \cap \ort{F}} \leq \frac{d}{d-s}\cdot\mathrm{radius}(E \cap \ort{F});
\]
\item\label{item:volbound}
the following volume bound holds:
\[
    \parenth{\frac{\vol[s]{K\cap F}}{\vol[s]{E\cap F}}}^\frac{1}{s} \leq
\frac{1}{(s + 1)^{\frac{d - s}{2s(d + 1)}}}
\sqrt{ \frac{d(d + 1)}{s(s + 1)} }
    \parenth{\frac{\volnf[s]{\Delta^s_J}}{\volnf[s]{\ball{s}}}}^\frac{1}{s},
\]
where $\Delta_J^s$ stands for the $s$-dimensional regular simplex circumscribed around $\ball{s}$;
\item\label{item:inclusion}
the following inclusion holds:
\[
E \cap F \subset  K \cap F\subset d \parenth{E \cap F}.
\]
\end{enumerate}
If additionally, $K$ is centrally symmetric, then
\begin{enumerate}[label=({S\arabic*})]
    \item\label{item:inradiussym}
    the inradius of $K \cap \ort{F}$  satisfies
\[
 \inrad \parenth{K \cap \ort{F}} \leq \sqrt{\frac{d}{d-s}}\cdot\mathrm{radius}(E \cap \ort{F});
\]
\item\label{item:volboundsym}
the following volume bound holds:
\[
    \parenth{\frac{\vol[s]{K\cap F}}{\vol[s]{E\cap F}}}^\frac{1}{s} \leq
\sqrt{\frac{d}{s}}
    \parenth{\frac{\volnf[s]{\cube^s}}{\volnf[s]{\ball{s}}}}^\frac{1}{s},
\]
where $\cube^s$ stands for the standard cube $[-1,1]^s$;
\item\label{item:inclusionsym}
the following inclusion holds:
\[
E \cap F \subset  K \cap F \subset \sqrt{d} \parenth{E \cap F}.
\]
\end{enumerate}
\end{thm}

Note that \ref{item:inclusion} and \ref{item:inclusionsym} are sharp: $\ball{d}$ is the John ellipsoid of $\Delta^d$ (resp., of $\cube^d$), and hence, it is the solution of \probellfixed{} for any $F$. If we now choose $F$ properly, then we have equality in these two inequalities.

On the other hand, \ref{item:volbound} and \ref{item:volboundsym} are identical to the best bounds that are known for sections of $K$ and its (unrestricted) John ellipsoid. See our discussion of optimality in \Href{Subsection}{sec:volratiotight}.

There is hardly any difference between the proofs of \ref{item:volbound}, \ref{item:inclusion} (and \ref{item:volboundsym}, \ref{item:inclusionsym}) and the corresponding statements for John ellipsoids. Everything reduces to the algebraic conditions on the projections of the contact vectors onto $F$. The proof of \ref{item:inradius} (and \ref{item:inradiussym}) however, is more intricate.

We prove \Hrefs{Theorems}{thm:genfixed}{thm:genrotation} in \Hrefs{Sections}{sec:firstorder}{sec:genproof}. In \Href{Subsection}{sec:locmaxrotationsufficient}, we state a certain converse of \Hrefs{Theorems}{thm:genfixed}{thm:genrotation}, that is, that the equations in this theorem ensure that $K$ is of \emph{locally} maximum volume among certain affine images of $K$ under some additional conditions.

In \Href{Section}{sec:containment}, we consider the problem of finding a homothety ratio such that enlarging the inner body in some sort of John position by this ratio, it will contain the outer body. In particular, we prove \ref{item:inclusion} in \Href{Theorem}{thm:ellfixedproperties}, and also show that no similar containment holds in the whole space $\Red$. 

We may refer to the notion of an isosceles triangle as a `semi-Euclidean' notion, since it is neither an affine notion nor purely Euclidean -- that is, a notion invariant under similarity alone.  In \Href{Subsection}{sec:rightcones}, we show an application of our results to isosceles triangles, and more generally, to right cones to illustrate how our language is capable of addressing semi-Euclidean questions. 

In \Href{Section}{sec:ellipsoids}, we discuss the problem of finding largest volume ellipsoids of revolution in a convex body, and in particular, we prove \Href{Theorem}{thm:ellfixedproperties}. We discuss the uniqueness of the solution to these John type problems in \Href{Section}{sec:uniqueness}, and finally, in \Href{Section}{sec:lowner}, we outline how some of the results change for the dual question (often referred to as Löwner's ellipsoid), where the smallest volume ellipsoid containing a convex body is considered. 
In \Href{Appendix}{sec:tightness_inrad_bound}, we show that \ref{item:inradius} of \Href{Theorem}{thm:ellfixedproperties} is tight. \Href{Appendix}{sec:major_lemma} contains the proof of a technical fact that we use in \Href{Section}{sec:ellipsoids}.

\subsection{Notation}
We denote by $\origin$ the origin of a vector space, and sometimes in a subscript, we specify the space, e.g. $\originmsp{\mathcal{L}}$ in a space $\mathcal{L}$ of matrices.
We use $\iprod{u}{v}$ to denote the standard \emph{inner product} of the vectors $u,v\in\Red$, and 
$\ball{d}$ for the closed Euclidean unit ball in $\Red$ centered at the origin.
We denote the \emph{positive hull} of a set $S$ in $\ReN$ by 
$\pos{S}=\set{\lambda x\st \lambda\geq 0, x\in S}.$
The \emph{polar} of a set $S$ is defined as
$S\pol=\set{p\in\ReN\st \iprod{x}{p} \leq 1\, \text{ for all } x\in S}$.
Note that for a positive cone $C$ (that is, the positive hull of some set), 
we have
$C\pol=\set{p\in\ReN\st \iprod{x}{p} \leq 0\, \text{ for all } x\in C}$.
The group of orthogonal transformations on $\Red$ is $O(d)$.

Recall that the \emph{polar decomposition} of a square matrix $A\in\Re^{d\times d}$ is a factorization of $A$ in the form  $A = R M,$ where $R$ is an orthogonal matrix and $M$ is positive semi-definite. Moreover, if $A$ is invertible, then this representation is unique ($M = \sqrt{\transpose{A}A}$ and $U = A \inv{M}$), where $\transpose{A}$ denotes the transpose of $A$.

\section{First-order necessary conditions}\label{sec:firstorder}

In this section, we reformulate \Hrefs{Theorems}{thm:genfixed}{thm:genrotation} in terms of the (non-)existence of a certain curve in a vector space. The difference between the proofs of the two theorems lies in the choice of the vector space in which the curve (non-)exists. This condition is then translated into a statement about the separation of a point and a convex cone in the corresponding space. This is a standard technique in proving John-type theorems; cf. \cite{ball1997elementary}.

First, we recall basic facts describing convex sets locally, and next, we introduce the vector space which parametrizes local affine transformations. In the final part of this section, we rephrase the existence of a curve in this vector space along which an `improved' position of $K$ in $L$ can be found in terms of separation of cones.

\subsection{The normal and the tangent cone}

\begin{dfn} Let $C$ be a convex body in $\ReN$ and $a\in\bd{C}$ be a boundary point.
The \emph{normal cone} 
$\ncone{K}{a}$ is the set consisting of the origin and all the outward normal vectors to $C$ at $a$. That is,
\[
\ncone{C}{a} = \left\{p\in\ReN \st \iprod{p}{z - a} \leq 0  \quad \forall z \in C \right\}.
\]
The \emph{tangent cone} $\tancone{C}{a}$ at $a$ is the closure of the cone  $\poscone{\parenth{C-a}}$.
\end{dfn}

The following facts are well-known, and follow from standard separation properties of convex sets \cite[Section 5.2]{hiriart2004fundamentals}:
\begin{lem}\label{lem:tangent_normal_cones_duality}
Consider a convex set $C$  in $\ReN$ and a point $a\in\cl{C}$. Then
\[
{\tancone{C}{a}} = \polarset{\parenth{\ncone{C}{a}}} 
\quad \text{and} \quad
\polarset{\parenth{\tancone{C}{a}}} = {\ncone{C}{a}}. 
\]
In particular, $h \in \inter \tancone{C}{a}$ if and only if
\[
 \iprod{h}{p} < 0
\] for all non-zero $p \in \ncone{C}{a}$.
\end{lem}
\begin{lem}\label{lem:tangent_cone_possible_directions}
Let $C$ be a convex body in $\ReN$, and $a \in C.$ Let $\gamma_t$ where $  t \in [0,1]$ be a curve in $\ReN$
with $\gamma_0 = a$ and 
\[\diffzero \gamma_t = h.\]
 Consider the following assertions:
\begin{enumerate}[label={(\roman*)}]
\item\label{item:intercone} 
    $h \in \inter \tancone{C}{a};$
\item\label{item:gammainbody} 
    $\gamma_t \in C$ for all $t \in [0, \tau]$ for some $\tau > 0;$
\item\label{item:incone}
    $h \in  \tancone{C}{a}$.
\end{enumerate}
Then \ref{item:intercone} implies \ref{item:gammainbody}, and \ref{item:gammainbody} implies \ref{item:incone}.
\end{lem}

\subsection{Operator spaces}

We use $\sqmatrix{d}$ and $\Symmatrix{d}$ to denote the space of operators and the space of self-adjoint operators from $\Red$ to $\Red$, respectively.
The \emph{inner product} on $\sqmatrix{d}$ is defined as $\iprod{A}{B}_{\sqmatrix{d}}=\tr{A\transpose{B}}$.

Define
\[
\rotpos = \braces{ (M ,S, g) \st M, S \in \sqmatrix{d},\ g \in \Red }.
\]
We endow 
$\rotpos$ with an inner product as follows:
\[
 \iprodsp{(M_1,S_1, g_1)}{(M_2,S_2, g_2)}
 =
\iprod{M_1}{M_2}_{\sqmatrix{d}} +
\iprod{S_1}{S_2}_{\sqmatrix{d}}
+\iprod{g_1}{g_2}.
\]

\subsection{A curve in \texorpdfstring{$\rotpos$}{R}}\label{subsec:curve_path}

First, we describe certain one parameter families of affine transformations of $\Red$. More specifically, we consider a differentiable path $\gamma: [0,t_0)\longrightarrow\rotpos$ for some $t_0>0$ defined by
 \begin{equation}
 \label{eq:path_in_m_space}
\gamma_t=(\id_{\Red} + tH, R_t, h_t),
 \end{equation}
where $h_0=\origin$, $R_t\in O(d)$ and $R_0=\id_{\Red}$.

It is well-known \cite[Proposition 3.24.]{hall2013lie} that by setting
\begin{equation}\label{eq:derivativeOfO}
W=\diffzero  R_t,
\end{equation}
we have that $W$ is a \emph{skew-symmetric} matrix, that is, $W=-\transpose{W}$, and conversely, for every skew-symmetric matrix $W$, there is a path $R_t$ in the orthogonal group satisfying \eqref{eq:derivativeOfO}.

Denoting $\diffzero h_t$ by $z$, we have
\begin{equation}\label{eq:derivativeOfgamma}
\diffzero \gamma_t=(H,W,z).
\end{equation}

For any $t$, we will regard $\gamma_t$ as a representation of the affine transformation
\begin{equation}\label{eq:atdef}
    A_t = R_t(\id_{\Red}+tH)+h_t
\end{equation}
on $\Red$. 

Throughout the section, $R_t, A_t, h_t, H, W, z$ denote the quantities defined so far, all derived from a differentiable path $\gamma_t$ described in \eqref{eq:path_in_m_space}.

Let $K$ and $L$ be convex bodies in $\Red$ with $K \subset L$. 
We are interested in two questions: first, how does the volume of $K$ change when we replace $K$ by $A_t(K)$? Second, under what conditions is $A_t(K)$ contained in $L$?

The first question is easy. We differentiate the determinant of the linear operator part of $A_t$:
\begin{equation}\label{eq:derivativeistrace}
  \diffzero\det (R_t(\id_{\Red}+tH))=  \trace{H} =
\end{equation}
\[  
  \iprodsp{\parenth{\id_{\Red}, \originmsp{\sqmatrix{d}}, \origin}}{\parenth{H, W, z}}=
\iprodsp{\parenth{\id_{\Red}, \originmsp{\sqmatrix{d}}, \origin}}{\diffzero \gamma_t}.
\]

Setting
\[
\conerotpos = \poscone{\parenth{\id_{\Red}, \originmsp{\sqmatrix{d}},\origin}},
\]
which is a ray in $\rotpos$, we conclude the following.
\begin{lem}\label{lem:volume_along_curve}
    If 
    \[
     \parenth{H, W, z} \in  - \inter\parenth{\polarset{\parenth{\conerotpos}}},
    \]
    then the volume of $A_t(K)$ is strictly increasing for sufficiently small $t>0$.
\end{lem}

Second, \Hrefs{Lemmas}{lem:tangent_normal_cones_duality}{lem:tangent_cone_possible_directions} provide a nice tool for characterization of local inclusions along curves.

\begin{lem}\label{lem:glmp_contact_point_along_curve}Let $L$ be a convex body in $\Red$ containing the origin in its interior, and let $v$ be a boundary point of $L$. Define the perturbation $v_t$ of $v$ by 
\[
v_t = A_t(v).
\]
Consider the following statements:
\begin{enumerate}
    \item\label{item:curve_admissible_linearized_in_primal} 
    \[
\iprodsp{ (H,W, z) }{ ( p \otimes v,   p \otimes v, p) } 
< 0;
    \]
    for all $p \in \polarset{L}$ with $\iprod{ p }{ v } = 1$.
    \item\label{item:curve_gammaadmissible}  	
    $v_t \in L$ for all $t \in (0, \tau]$ and some $\tau > 0$.
    \item\label{item:curve_admissible_linearized_in_primalweak} 
    \[
\iprodsp{ (H,W, z) }{ ( p \otimes v,   p \otimes v, p) }
  \leq 0;
    \]
    for all $p \in \polarset{L}$ with $\iprod{ p }{ v } = 1$.
\end{enumerate}
Then \eqref{item:curve_admissible_linearized_in_primal} implies \eqref{item:curve_gammaadmissible}, and \eqref{item:curve_gammaadmissible} implies \eqref{item:curve_admissible_linearized_in_primalweak}.
\end{lem}

\begin{proof}
By direct computation,
\[
\diffzero v_t = \parenth{H + W}  v + z.
\]
Thus, for any $p \in \Red$, we have
\[
\iprod{ \diffzero v_t  }{ p } = \iprod{ \parenth{H + W }  v + z }{ p } = \iprod{ H   v }{ p } + \iprod{ W  v }{ p } + \iprod{ z }{ p } =\]
\[
\iprodsp{ (H,W, z) }{ ( p \otimes v,   p \otimes v, p) }.
\]


The result follows from the descriptions of the tangent cone in \Href{Lemma}{lem:tangent_normal_cones_duality} and \Href{Lemma}{lem:tangent_cone_possible_directions}. 
\end{proof}

We denote the set of \emph{contact points} of $K$ and $L$ by
\[
\contact=\bd{K}\cap\bd{L},
\]
and, from this point onward, assume that $K$ contains the origin in its interior.

\begin{cor}\label{cor:iprodinclusion}
 If 
  \begin{equation}\label{eq:inwardconditionIprod}
\iprodsp{ (H,W, z) }
{ ( p \otimes v,   p \otimes v, p) }
<0
  \end{equation}
  for all $v \in \contact$ and  for all $p \in \polarset{L}$ with $\iprod{ p }{ v } = 1$,
 then for a sufficiently small $\epsilon>0$, the set $A_t K + h_t$ is contained in $\inter{L}$ for all $t\in(0,\epsilon)$.
\end{cor}
\begin{proof}
By \Href{Lemma}{lem:glmp_contact_point_along_curve},
for any $v \in\contact$,  $A_t(v)$ is in $\inter{K}$ for all $t\in(0,\epsilon_v)$  with some $\epsilon_v > 0$ depending on $v$. A standard compactness argument completes the proof.
\end{proof}


Define the set of \emph{contact operators} by
\[
\ucontopcr= \braces{
 \parenth{ p \otimes v,  p \otimes v, p } \st
v\in \contact, p\in\polarset{L}, \iprod{v}{p}=1},
\]
and the \emph{contact operator cone} by
\[
\uconttensorcr= 
\poscone{ \ucontopcr}.
\]

With this notation, \Href{Corollary}{cor:iprodinclusion} can be restated as follows.
\begin{cor}\label{cor:local_inclusion_along_curve}
 If 
  \begin{equation}\label{eq:inwardconditioncone}
 (H, W, z)  \in \inter{\parenth{\polarset{\uconttensorcr}}},
  \end{equation}
 then for a sufficiently small $\epsilon>0$, the set $A_t K + h_t$ is contained in $\inter{L}$ for all $t\in(0,\epsilon)$.
\end{cor}

\begin{dfn}\label{dfn:locmax}
Let $K\subseteq L$ be convex bodies in $\Red$.
We say that $K$ is of \emph{locally maximum volume among elements of $\rpositions{K}$} (resp., \emph{$\apositions{K}$}) contained in $L$, if there is no $t_0>0$ and differentiable curve $A_t:[0,t_0)\longrightarrow \rposition$ (resp., $A_t:[0,t_0)\longrightarrow \aposition$) such that $A_0=\id_{\Red}$, $\diffzero \volnf{A_t(K)}>0$ and $A_t(K)\subset \inter{L}$ for all $t\in(0,t_0)$.
\end{dfn}

For some fixed subspace $F$ of $\Red$, we set
\[
\gammaderivatives= \bigg\{ (H,W,z) \in \rotpos \st H\ \text{ is an } F\text{-operator, } W=-\transpose{W},\ z \in \Red\bigg\}.
\]

\begin{lem}[Local optimality in $\rpositions{K}$]\label{lem:coneconditionNoImprovementRotation}
    If $\origin\in\inter{K}$ and $K$ is of locally maximum volume among elements of $\rpositions{K}$ contained in $L$, then
    \begin{equation*}
    \parenth{\conerotpos + \ort{\gammaderivatives}}
    \cap \uconttensorcr \neq \{\originmsp{\rotpos}\}.
    \end{equation*} 
\end{lem}
\begin{proof}
By \eqref{eq:derivativeOfO} and \eqref{eq:derivativeOfgamma}, we have that if $(H,W,z)$ is in $\gammaderivatives$, then there is a curve $\gamma_t$ in $\rotpos$ with $\diffzero\gamma_t=(H,W,z)$, and $\gamma_0=(\id_{\Red},\originmsp{\sqmatrix{d}},\origin)$ such that $A_t$ defined by \eqref{eq:atdef} is a perturbation of $K$ with $A_t(K)\in\rpositions{K}.$

Combining this observation with \Href{Lemma}{lem:volume_along_curve} and \Href{Corollary}{cor:local_inclusion_along_curve}, we obtain the following.

If $K\subseteq L$ are convex bodies in $\Red$ with $\origin\in\inter L$, and
\begin{equation*}
\gammaderivatives \cap \inter{\parenth{\polarset{\uconttensorcr}}} \cap  \inter\parenth{-\polarset{\parenth{\conerotpos}}}
\neq\{ \originmsp{\rotpos}\},
\end{equation*}
then there is a curve $\gamma$ showing that $K$ is not of locally maximum volume among elements of $\rpositions{K}$ contained in $L$.

Equivalently, if $K$ is of locally maximum volume among elements of $\rpositions{K}$ contained in $L$, then
\begin{equation}\label{eq:coneconditionforimprovement}
\gammaderivatives \cap  
\inter\parenth{-\polarset{\parenth{\conerotpos}}} \cap 
\inter{\parenth{\polarset{\uconttensorcr}}}
=\{\originmsp{\rotpos}\}.
\end{equation}

Notice that $\gammaderivatives$ is a linear subspace in $\rotpos$, and the other two sets on the left-hand side are open cones in $\rotpos$.
By basic separation properties of cones in Euclidean space, the proof of \Href{Lemma}{lem:coneconditionNoImprovementRotation} is complete.
\end{proof}

What changes if we are not allowed to use a rotation in the setting of \Href{Theorem}{thm:genfixed}? In this case, the curve $A_t$ is in the subspace $\aposition$, or, equivalently, $\gamma_t$ is such that $R_t=\id_{\Red}$ for all $t$, which is equivalent to the condition that $\diffzero\gamma_t$ takes all values in 
\[
\gammaderivatives_0= \bigg\{ (H,\originmsp{\sqmatrix{d}} ,z) \in \rotpos \st H\ \text{ is an } F\text{-operator, }\ z \in \Red\bigg\}.
\]
Thus, \eqref{eq:coneconditionforimprovement} changes to the following.

If $K$ is of locally maximum volume among elements of $\apositions{K}$ contained in $L$, then
\begin{equation*}
\gammaderivatives_0 \cap  
\inter\parenth{-\polarset{\parenth{\conerotpos}}} \cap 
\inter{\parenth{\polarset{\uconttensorcr}}}
=\{\originmsp{\rotpos}\}.
\end{equation*}

This yields the following.
\begin{lem}[Local optimality in $\apositions{K}$]\label{lem:coneconditionNoImprovementFixedf}
    If $\origin\in\inter{K}$ and $K$ is of locally maximum volume among elements of $\apositions{K}$ contained in $L$, then
    \begin{equation*}
    \parenth{\conerotpos + \ort{\gammaderivatives_0}}
    \cap \uconttensorcr \neq \{\originmsp{\rotpos}\}.
    \end{equation*}    
\end{lem}

In the next two sections, we prove \Hrefs{Theorems}{thm:genrotation}{thm:genfixed}, using
\Hrefs{Lemmas}{lem:coneconditionNoImprovementRotation}{lem:coneconditionNoImprovementFixedf}, respectively.

\section{Proof of \Hrefstitle{Theorems}{thm:genrotation}{thm:genfixed}}\label{sec:genproof}

We fix a non-trivial subspace $F$ of $\Red$ of dimension $s.$
We use $P_F$ to denote the orthogonal projection onto $F$.
Set
    \[
     \mathcal{Q}_F=\braces{Q\in\lin{\Red} \st P_FQP_F=\originmsp{F}, \rank{P_{\ort{F}}QP_{\ort{F}}}=d - s, \tr{P_{\ort{F}}QP_{\ort{F}}}=0},
    \]
where $\originmsp{F}$ is the zero operator on $F$.  In coordinates, if we consider an orthonormal basis of $F$ and extend it to an orthonormal basis of $\Red$, then the matrix of $Q\in\mathcal{Q}_F$ in this basis is of the form 
\[
  Q=\begin{pmatrix} \originmsp{s \times s}&Q_1\\ Q_2&Q_3\end{pmatrix},
  \]
where $\originmsp{s \times s}$ is the $s\times s$ zero matrix,  
$Q_1$ and $\transpose{Q_2}$
are  $s \times (d-s)$ matrices, and $Q_3$ is a full rank $(d-s)\times(d-s)$  matrix of zero trace.

\begin{lem}\label{lem:ort_complement_derriv_space}
The following identities hold:
\[
\ort{\gammaderivatives} = \mathcal{Q}_F\times\Symmatrix{d}\times\{\origin\},
\quad \text{and} \quad 
\ort{\gammaderivatives_0} = \mathcal{Q}_F\times\sqmatrix{d}\times\{\origin\}.
\]
\end{lem}

\begin{proof}
Let $\parenth{M,S,g}\in \rotpos$. 
If $\parenth{M,S,g} \in \ort{\mathcal{D}}$, then
\[
\iprodsp{\parenth{M,S,g}}
{\parenth{\originmsp{\sqmatrix{d}},
\originmsp{\sqmatrix{d}},z}}
=
\iprod{g}{z}=0
\]
for all $z\in\Red$. Hence, $g=\origin$.

The orthogonal complement of the space of skew-symmetric matrices in $\sqmatrix{d}$ is $\Symmatrix{d}$. 
Thus, $S\in \Symmatrix{d}$.

Finally, let $M$ and an $F$-matrix $H$ be of the form:
\[
H=\begin{pmatrix} {H_1}& \origin_{s \times (d-s)}\\ \origin_{ (d-s) \times s} & \mu \id_{d-s}\end{pmatrix}
\quad \text{and} \quad
M=\begin{pmatrix} {M_1}& {M_2}\\ {M_3} & {M_4}\end{pmatrix}
\]
for some non-zero constant $\mu$,
where
$H_1$ and $M_1$ are $s\times s$ matrices, $M_4$ is a $(d-s)\times(d-s)$ matrix, and $M_2$ and $\transpose{M_3}$ are $s\times(d-s)$ matrices.

Then
\[
\iprodsp{\parenth{M,S,\origin}}
{\parenth{H,\originmsp{\sqmatrix{d}},z}}
=
\tr{M \transpose{H}}=\tr{M_1 \transpose{H_1}}+\mu \trace{M_4}.
\]

As $H_1$ and $\mu$ can be chosen arbitrarily,
\[
\iprodsp{\parenth{M,S,\origin}}{\parenth{H,\originmsp{\sqmatrix{d}},z}}=0
\]
if and only if $M_1=\origin_{s\times s}$ and $\trace{M_4}=0$, while there are no constraints on $M_2$ and $M_3$. This is equivalent to $M\in\mathcal{Q}_F$. 

Thus, 
\[
\ort{\gammaderivatives} = \mathcal{Q}_F\times\Symmatrix{d}\times\{\origin\}.
\]

The second statement of the lemma follows analogously.
\end{proof}

\subsection{Proofs of \Hrefstitle{Theorem}{thm:genrotation}{thm:genfixed}}
\noshow{
In this subsection, we prove variants of these two theorems without a translation $z$, and as a price, we do not obtain \eqref{eq:answerGLMPfixedsubspace_zerosumv}.

The following result differs slightly from \Href{Theorem}{thm:genrotation}.

\begin{thm}\label{thm:locmaxrotation}
 Let $K$, $L$ and $F$ satisfy the conditions of \Href{Theorem}{thm:genrotation}, and assume that the origin is in the interior of $K$, and that $K$ is of locally maximum volume among elements of $\rpositions{K}$ contained in $L$.
 
 Then we have contact pairs of $K$ and $L$ and weights satisfying \eqref{eq:answerGLMPfixedsubspace_decomposition}, \eqref{eq:answerGLMPfixedsubspace_zerosum}, \eqref{eq:answerGLMPfixedsubspace_trace}, and \eqref{eq:answerGLMProtation}.
\end{thm}
}

\begin{proof}[Proof of \Href{Theorem}{thm:genrotation}]
\Hrefs{Lemmas}{lem:coneconditionNoImprovementRotation}{lem:ort_complement_derriv_space} yield that 
   \begin{equation*}
    \parenth{\poscone\{(\id_{\Red},\originmsp{\sqmatrix{d}},\origin)\} + \mathcal{Q}_F\times\Symmatrix{d}\times\{\origin\}}
    \cap \uconttensorcr \neq \{\originmsp{\rotpos}\}.
    \end{equation*} 

Thus, there are matrices $Q \in \mathcal{Q}_F$  and $S \in \Symmatrix{d}$ such that 
$\parenth{\id_{\Red} + Q, S, \origin} \in  \uconttensorcr$.
 That is, there are contact pairs 
 $(v_1, p_1),\dots,(v_m, p_m)$ of $K$ and $L,$  and weights $\alpha_1,\dots,\alpha_m > 0$ such that  $\sum\limits_{i\in [m]} \alpha_i p_i =\origin$
 and the matrix 
$M=\sum\limits_{i \in [m]} \alpha_i p_i\otimes v_i$
satisfies the conditions
\begin{itemize}
\item $ P_F M P_F= P_F(\id_{\Red} +Q)P_F = P_F \id_{\Red} P_F = P_F$, 
\item $M = \transpose{M}$, 
\item $\trace{M} =\tr{\id_{\Red} +Q} = \trace{\id_{\Red}} = d.$
\end{itemize}
The proof of \Href{Theorem}{thm:genrotation} is complete.
\end{proof}

\noshow{
     
}
The proof of \Href{Theorem}{thm:genfixed} is essentially identical to the previous one.

\begin{proof}[Proof of \Href{Theorem}{thm:genfixed}]
\Hrefs{Lemmas}{lem:coneconditionNoImprovementFixedf}{lem:ort_complement_derriv_space} yield that 
   \begin{equation*}
    \parenth{\poscone\{(\id_{\Red},\originmsp{\sqmatrix{d}},\origin)\} + \mathcal{Q}_F\times\sqmatrix{d}\times \{\origin\}}
    \cap \uconttensorcr \neq \{\originmsp{\rotpos}\}.
    \end{equation*} 

Thus, there are matrices $Q \in \mathcal{Q}_F$  and $S \in \sqmatrix{d}$ such that 
$\parenth{\id_{\Red} + Q, S, \origin} \in  \uconttensorcr$.
 That is, there are contact pairs 
 $(v_1, p_1),\dots,(v_m, p_m)$ of $K$ and $L,$  and weights $\alpha_1,\dots,\alpha_m > 0$ such that  $\sum\limits_{i\in [m]} \alpha_i p_i =\origin$
 and 
 the matrix 
$M=\sum\limits_{i \in [m]} \alpha_i p_i\otimes v_i$
satisfies the conditions
\begin{itemize}
\item $ P_F M P_F= P_F(\id_{\Red} +Q)P_F = P_F \id_{\Red} P_F = P_F$, and
\item $\trace{M} =\tr{\id_{\Red} +Q} = \trace{\id_{\Red}} = d.$
\end{itemize}
The proof of \Href{Theorem}{thm:genfixed} is complete.
\end{proof}

\subsection{Rotations}
This subsection is a small detour in which we clarify the geometry hidden behind the symmetry condition \eqref{eq:answerGLMProtation}. The proofs of our other results do not rely on it.

\begin{thm}
    Let $K$ and $L$ be convex bodies in $\Red$ such that the origin is in the interior of $L \cap K,$ and  there are no 
    contact pairs $(v_1, p_1),\dots, (v_m, p_m)$ of $L$ and $K$ and corresponding
    positive weights $\alpha_1,\dots,\alpha_m$
satisfying identity \eqref{eq:answerGLMProtation}.
Then for some $\epsilon>0$, there is a differentiable path $R_t,$ $t \in (0, \epsilon]$ of orthogonal transformations such that 
$R_t K \subset \inter{L}$ for all  $t \in (0,\epsilon].$
\end{thm}
\begin{proof}
There is nothing to prove if the set of contact pairs is empty.
We assume that the cone $\operatorname{Cone}_{\mathrm{op}}=\poscone{\braces{p\otimes v \st (v,p) \text{ is a contact pair of } K  \text{ and } L}}$ is non-empty. Clearly, it is a closed set.

Then the negation of identity \eqref{eq:answerGLMProtation} is equivalent to 
\[
\Symmatrix{d} \cap \operatorname{Cone}_{\mathrm{op}} = \braces{\origin_{\sqmatrix{d}}}.
\]
Thus, there is a non-zero matrix $W$ separating the closed convex cone $\operatorname{Cone}_{\mathrm{op}}$ from $\Symmatrix{d}$ in the following sense: for every nonzero $M_1\in\Symmatrix{d}$ and $M_2\in\operatorname{Cone}_{\mathrm{op}}$, we have $\iprod{M_1}{W}_{\sqmatrix{d}}\leq 0<\iprod{M_2}{W}_{\sqmatrix{d}}$.
Since $\Symmatrix{d}$ is a subspace, $W \in \ort{\parenth{\Symmatrix{d}}},$
that is, $W$ is skew-symmetric. 

Set $R_t = e^{tW}$ and  $H = 0$  in 
\eqref{eq:path_in_m_space}. Clearly, $R_t, t \in [0,1]$ is a differentiable path of  orthogonal matrices in $\sqmatrix{d}$ satisfying \eqref{eq:derivativeOfO}.
By direct computation,
\[
\diffzero \parenth{R_t v} =  W  v.
\]
The result follows from the descriptions of the tangent cone in \Href{Lemma}{lem:tangent_normal_cones_duality} and \Href{Lemma}{lem:tangent_cone_possible_directions}, and standard compactness arguments. 

\end{proof}

\subsection{The condition is sufficient for local extremality -- in some cases}\label{sec:locmaxrotationsufficient}
According to \Href{Theorem}{thm:genfixed}, the local maximality of the volume of $K$ implies certain relations on contact pairs of $K$ and $L$.

The converse is not true in general. If we consider a regular simplex centered at the origin as $K$, and $L=-d \cdot K$, then there are contact pairs satisfying equations \eqref{eq:answerGLMPfixedsubspace_decomposition}--\eqref{eq:answerGLMProtation}, and yet, clearly, $K$ is not a locally largest volume $F$-position of $K$ in $L$, for $F=\Red$.

However, as in the classical setting, the log-concavity of the determinant (Minkowski's determinant inequality) yields the sufficiency of our conditions for positions $AK  + a$ with positive definite $A$. Denote by
\[
 \mathcal{A}_{F}^{+} (S) =\braces{AS+z\st A \text{ is a positive definite } F\text{-operator, } z\in\Red},
\]
the \emph{positive $F$-positions} of $S$.
\begin{thm}\label{thm:locmaxrotationsufficient}
 Let $K$ and $L$ be convex bodies in $\Red$ and $F$ a linear subspace in $\Red$. Assume that $K\subseteq L$ and there are contact points of $K$ and $L$ and weights satisfying \eqref{eq:answerGLMPfixedsubspace_decomposition}, \eqref{eq:answerGLMPfixedsubspace_zerosum} and \eqref{eq:answerGLMPfixedsubspace_trace}. Then $K$ is of locally maximum volume among elements of $ \mathcal{A}_{F}^{+} (K)$  contained in $L$.
\end{thm}
We leave the proof to the reader.

\section{Containment}\label{sec:containment}
Let us recall how the authors of \cite{GLMP04} derived 
\Href{Theorem}{thm:GLMP_inclusion_in_F} from \Href{Theorem}{thm:genfixed} in the case \( F = \Red \). 
Note that the identities \eqref{eq:answerGLMPfixedsubspace_decomposition}, 
\eqref{eq:answerGLMPfixedsubspace_zerosum}, and 
\eqref{eq:answerGLMPfixedsubspace_trace} in \Href{Theorem}{thm:genfixed} 
can be rewritten with respect to an arbitrary point in the interior of \( K \); in other words, 
we may choose the origin arbitrarily within the interior of \( K \). 

In \cite{GLMP04} it is shown that in the case \( F = \Red \), 
one can choose the origin so that, in addition to the identities 
\eqref{eq:answerGLMPfixedsubspace_decomposition}, 
\eqref{eq:answerGLMPfixedsubspace_zerosum}, and 
\eqref{eq:answerGLMPfixedsubspace_trace}, the following identity also holds
\[
\sum\limits_{i \in [m]} \alpha_i v_i = 0.
\]

As we will show, in our setting, the following relation suffices to prove the containment in \Href{Theorem}{thm:GLMP_inclusion_in_F}:
\begin{equation}\label{eq:desired_fixedsubspace_zerosumv}
\sum\limits_{i \in [m]} \alpha_i P_F v_i = \origin.
\end{equation}

\subsection{Inclusion for contact polytopes}
\label{subsec:inclusion_contact_polytopes}
Let $K$ and $L$ be convex bodies in $\Red$ such that 
$K \subset L$ and $K$ contains the origin in its interior. 
Let $F$ be a non-trivial subspace of $\Red.$

Assume that there are contact pairs
$(v_1, p_1), \dots, (v_m, p_m)$ of $K$ 
and $L,$ and weights $\alpha_1, \dots, \alpha_m$ such that
identities  \eqref{eq:answerGLMPfixedsubspace_decomposition} -- \eqref{eq:answerGLMPfixedsubspace_trace} hold, and moreover, identity \eqref{eq:desired_fixedsubspace_zerosumv} holds.

Denote
$K_{\text{in}} = \conv{\braces{v_1, \dots, v_m}}$ and 
$L_{\text{out}} = \bigcap\limits_{i \in [m]} \braces{x \in \Red \st 
\iprod{p_i}{x} \leq 1}.$

We call $K_{\text{in}}$ and $L_{\text{out}}$ the \emph{inner} and \emph{outer contact polytope}, respectively, corresponding to the contact pairs $(v_i,p_i)$. Clearly,
\[
K_{\text{in}} \subset K \subset L \subset L_{\text{out}}.
\] 

%

\begin{lem}
\label{lem:inclusion_contact_polytopes}
With the above notation and assumptions, we have
\[
 L_{\mathrm{out}} \cap F \subset -d\cdot P_F K_{\mathrm{in}}.
\]
\end{lem}
\begin{proof}
Equation \eqref{eq:answerGLMPfixedsubspace_decomposition} is equivalent to
\begin{equation}
\label{eq:transposeGLMPfixedsubspace_decomposition}
P_F {\parenth{\sum\limits_{i \in [m]} \alpha_i v_i \otimes p_i}} P_F = 
P_F \transpose{\parenth{\sum\limits_{i \in [m]} \alpha_i p_i \otimes v_i}} P_F = P_F.
\end{equation}
 Let $x$ be a point in $L_{\text{out}} \cap F$, then $\iprod{p_i}{x}\leq 1$ for all $i \in [m].$ 
 \[
 -x \stackrel{x \in F}{=} - P_F x 
 \stackrel{}{=} 
   -P_F\sum\limits_{i \in [m]} \alpha_i v_i\iprod{ p_i}{P_F x}
   \stackrel{x \in F}{=}
-P_F\sum\limits_{i \in [m]} \alpha_i v_i\iprod{ p_i}{ x}
      \stackrel{\eqref{eq:desired_fixedsubspace_zerosumv}}{=}
\]
\[
\sum\limits_{i \in [m]}  \alpha_i (1-\iprod{ p_i}{x}) P_F v_i\in
 \sum\limits_{i \in [m]} \alpha_i \cdot P_F K_{\text{in}} \cdot(1-\iprod{ p_i}{x})
   \stackrel{\iprod{p_i}{x} \leq 1}{=}
\]
\[
  \parenth{\sum\limits_{i \in [m]} \alpha_i \cdot(1-\iprod{p_i}{x})}  P_F K_{\text{in}}
 \stackrel{\eqref{eq:answerGLMPfixedsubspace_trace}}{=}
 \parenth{ d-  \iprod{\sum\limits_{i \in [m]}\alpha_i p_i}{x}}P_F K_{\text{in}} 
  \stackrel{\eqref{eq:answerGLMPfixedsubspace_zerosum}}{=} d \cdot P_F K_{\text{in}}.
 \]
\end{proof}
\subsection{Search for a good center}

Now, we show that the containment $P_F K \subset L \cap F$ guarantees the existence of a translation, after which \eqref{eq:desired_fixedsubspace_zerosumv} holds.

\begin{lem}\label{lem:goodcenter}
Let $K\subseteq L$ be convex bodies in $\Red$ such that the origin is in the interior of $K$, and  there are
\begin{enumerate}
     \item an $s$-dimensional subspace $F;$
    \item contact pairs $(v_1, p_1),\dots, (v_m, p_m)$ of $L$ and $K;$
    \item positive weights $\alpha_1,\dots,\alpha_m$
\end{enumerate} 
satisfying identities  \eqref{eq:answerGLMPfixedsubspace_decomposition}, \eqref{eq:answerGLMPfixedsubspace_zerosum}, and \eqref{eq:answerGLMPfixedsubspace_trace}.
Additionally, assume  $P_F K \subset L \cap F.$

Then there exists a point $z$ in the relative interior of $L \cap F$, contact pairs 
$\{ (v_1^\prime, p_1^\prime), \dots, (v_m^\prime, p_m^\prime) \}$ of $K - z$ and $L - z$, and weights $c_1,\dots,c_m>0$ such that
\begin{equation*}
{\sum_{i \in [m]} c_i p_i^\prime \otimes v_i^\prime } = 
\sum_{i \in [m]} \alpha_i p_i \otimes v_i , \quad \text{and} \quad
\sum_{i \in [m]} c_i P_F v_i^\prime = \sum_{i \in [m]} c_i p_i^\prime=\origin .
\end{equation*}
\end{lem}
\begin{proof}
This is a purely linear algebraic argument.

We aim at finding a point $z$ in  the relative interior of $L \cap F$ such that when we shift the origin to $z$, the shifted contact pairs $(v_i - z, p_i')$ satisfy the required identities for $K - z$ and $L - z$ with appropriate normal vectors $p_i'$.

First, observe how a contact pair transforms under translation. A pair $(v, p)$ is a contact pair of $K$ and $L$ if and only if, $v$ belongs to the intersection of the boundaries of $K$ and $L$, and the half-space
\[
H_p^{-} = \braces{ x \in \Red \st \iprod{ x }{ p } \leq \iprod{ v }{ p } = 1 }
\]
contains $L$.

Under translation by $z$, the half-space becomes
\[
H_p^{-} - z = \braces{ x - z \in \Red \st \iprod{ x }{ p } - \iprod{ z }{ p } \leq 1 - \iprod{ z }{ p } } = \braces{ y \in \Red \st \iprod{ y }{ \frac{ p }{ 1 - \iprod{ z }{ p } } } \leq 1 }.
\]
Thus, provided $\iprod{ z }{ p } < 1$, the pair $(v - z, p' )$, where $p' = \dfrac{ p }{ 1 - \iprod{ z }{ p } }$, is a contact pair of $K - z$ and $L - z$.

For every $i \in [m]$, define
\[
v_i( z ) = v_i - z, \quad \text{and} \quad p_i( z ) = \frac{ p_i }{ 1 - \iprod{ z }{ p_i } }.
\]
To satisfy the identity $\sum\limits_{ i \in [m] } c_i( z ) p_i( z ) =\origin $, we set
\[
c_i( z ) = \alpha_i \parenth{ 1 - \iprod{ z }{ p_i } }.
\]
Then,
\[
c_i( z ) p_i( z ) = \alpha_i \parenth{ 1 - \iprod{ z }{ p_i } } \frac{ p_i }{ 1 - \iprod{ z }{ p_i } } = \alpha_i p_i,
\]
which implies
\[
\sum_{ i \in [m] } c_i( z ) p_i( z ) = \sum_{ i \in [m] } \alpha_i p_i =\origin .
\]

Set $M =  \sum\limits_{i \in [m]} \alpha_i p_i\otimes v_i.$
Next, we verify the identity
$\sum\limits_{ i \in [m] } c_i( z ) p_i( z ) \otimes v_i( z ) = M$. Indeed,
\[
\sum_{ i \in [m] } c_i( z ) p_i( z ) \otimes v_i( z ) = \sum_{ i \in [m] } \alpha_i \parenth{ 1 - \iprod{ z }{ p_i } }  \frac{ p_i }{ 1 - \iprod{ z }{ p_i } }
 \otimes ( v_i - z ) 
 = \sum_{ i \in [m] } \alpha_i  p_i \otimes  ( v_i - z )=
\]
\[
\sum_{ i \in [m] } \alpha_i p_i \otimes v_i -
\sum_{ i \in [m] } \alpha_i p_i \otimes z = 
M - \left( \sum_{ i \in [m] } \alpha_i p_i \right ) \otimes  
z
\stackrel{\eqref{eq:answerGLMPfixedsubspace_zerosum}}{=} M - \origin \otimes z  = M.
\]

Finally, we need to satisfy $P_F \parenth{ \sum\limits_{ i \in [m] } c_i( z ) v_i( z )} =\origin $. Computing, we have
\begin{align*}
P_F \parenth{
\sum_{ i \in [m] } c_i( z ) v_i( z )} &= \sum_{ i \in [m] } \alpha_i \parenth{ 1 - \iprod{ z }{ p_i } } P_F( v_i - z ) \\
&= \sum_{ i \in [m] } \alpha_i P_F v_i - \sum_{ i \in [m] } \alpha_i \iprod{ z }{ p_i } P_F v_i - P_F z \sum_{ i \in [m] } \alpha_i + 
P_F z \sum_{ i \in [m] } \alpha_i \iprod{ z }{ p_i }.
\end{align*}
Using the identities $\sum\limits_{ i \in [m] } \alpha_i p_i =\origin $ and 
$M = \sum\limits_{ i \in [m] } \alpha_i p_i \otimes v_i$, we can simplify:
\begin{gather*}
\sum\limits_{ i \in [m] } c_i( z ) P_F v_i( z ) = 
\sum\limits_{ i \in [m] } \alpha_i P_F v_i - 
P_F \parenth{\sum\limits_{ i \in [m] } \alpha_i v_i \otimes p_i} z - 
P_F  z \sum\limits_{ i \in [m] } \alpha_i +
P_F z \iprod{ z }{ \sum\limits_{ i \in [m] } \alpha_i p_i } \\
= \sum\limits_{ i \in [m] } \alpha_i P_F v_i - P_F \transpose{M} z - P_F z 
\parenth{ \sum\limits_{ i \in [m] } \alpha_i} + P_F z \cdot 0 
= \parenth{\sum\limits_{ i \in [m] } \alpha_i P_F v_i} - P_F \transpose{M} z - d P_F z.
\end{gather*}

Set
\[
z =  P_F \parenth{\frac{ \sum\limits_{ i \in [m] } \alpha_i v_i }{ d + 1 }}.
\]
Since by our assumption $P_F K \subset L \cap F,$ and  since $\sum\limits_{ i \in [m] } \alpha_i=d$, and $\origin\in\inter{K}$, we conclude that $z$ is in the relative interior of ${L\cap F}$.
Thus, $\sum\limits_{ i \in [m] } c_i( z ) P_F v_i( z ) = \origin$, as desired.
\end{proof}

\noshow{
    
}

\subsection{Proof of \Hreftitle{Theorem}{thm:GLMP_inclusion_in_F}}
We use \Href{Lemma}{lem:goodcenter} to find a ``good center'' $z.$
The inclusion follows from \Href{Lemma}{lem:inclusion_contact_polytopes} and the inclusion 
$K_{\text{in}} \subset K \subset L \subset L_{\text{out}},$ where $K_{\text{in}}$ and $L_{\text{out}}$ are as defined in 
\Href{Subsection}{subsec:inclusion_contact_polytopes}. In the case when $K$ is symmetric with respect to $F,$ the inclusion $P_F K \subset L \cap F$ holds.

\subsection{Application: right cones}\label{sec:rightcones}
We define a \emph{right cone} in $\Red$ as the convex hull of a $(d-1)$-dimensional compact, convex set $C_b$ (the \emph{base}) in a hyperplane $H$ and a point $a$ (the \emph{apex}) not in $H$, whose orthogonal projection $b$ (the \emph{foot point}) onto $H$ is in $C_b$. The line through $a$ orthogonal to $H$ is the \emph{axis} of the cone.

As an application of \Href{Lemma}{lem:inclusion_contact_polytopes}, we obtain the following result, whose proof is essentially the same as that of \Href{Theorem}{thm:GLMP_inclusion_in_F}:
\begin{prp}
Let $K$ be a convex body and $C_b$ a $(d-1)$-dimensional compact, convex set in $\Red$, and $b\in C_b$. Consider those right cones $C$ contained in $K$, whose base is homothetic to $C_b$ with a homothety that takes $b$ to the foot point of $C$. Assume that $C_0$ is of maximum volume in this family of right cones, and denote its axis by $\ell$. Then there is a point $z$ in the relative interior of $\ell\cap C$ such that
\[
(K-z)\cap\ell\subseteq d(C_0-z)\cap \ell.
\]
\end{prp}

In particular, if $K$ is a convex body on the plane, and $\Delta$ is the largest area isosceles triangle in $K$ with a given direction of its axis of symmetry $\ell$, then with the right choice of $z\in\Delta\cap\ell$, we have
$(K-z)\cap\ell\subseteq 2(\Delta-z)\cap \ell$. Note that the factor 2 in this formula is optimal, as shown by the example when $K$ is a square and $\ell$ is parallel to one of its diagonals.

\subsection{No containment in the whole space}\label{sec:nocontainment}
Contrasting \Href{Theorem}{thm:GLMP_inclusion_in_F}, we show that there is no meaningful inclusion in $\ort{F}$. Surprisingly, the example is provided by ellipsoids.

\begin{prp}\label{prp:badellipsoid}
    Fix an integer $0 \leq s \leq d - 2$ and a positive real $\lambda$. Let $E$ be an origin-centered ellipsoid in $\Red$ whose principal semi-axes have lengths 
    $\lambda, \lambda^2, \ldots, \lambda^d$. Let $E^\prime \in \apositions{\ball{d}}$ be the solution to 
    \Href{Problem}{prob:ellrotation}, and let the $s$-dimensional subspace $F$ be the axis of $E^\prime$. Then 
    \[
    \frac{\volnf[d]{E'}}{\volnf[d]{E}} \leq \frac{1}{\lambda}
    \quad \text{and} \quad 
    \forall x \in \Red, \quad E \cap \ort{F} \not\subset x + \lambda \cdot \inter{P_{\ort{F}} E^\prime}.
    \]
\end{prp}

We need the following observation:

\begin{lem}\label{lem:contaiment_ell_semi_axes}
Let $E_1 \subseteq E_2 \subset \Red$ be two ellipsoids. Let the principal semi-axes of $E_1$ be $\lambda_1 \leq  \ldots \leq \lambda_d$, and those of $E_2$ be $\mu_1 \leq \ldots \leq \mu_d$. Then $\lambda_i \leq \mu_i$ for all $i \in [d]$.
\end{lem}

\begin{proof}
By symmetry, we may assume that both ellipsoids $E_1$ and $E_2$ are centered at the origin. Let $L_i$ be the $i$-dimensional subspace spanned by the principal semi-axes of $E_1$ with lengths $\lambda_{d-i+1}, \dots, \lambda_d$. The Courant--Fischer min-max theorem \cite[Theorem 4.2.6]{horn2012matrix} yields $\lambda_{d-i+1} \leq \mu_{d-i+1}$, concluding the proof.
\end{proof}

\begin{proof}[Proof of \Href{Proposition}{prp:badellipsoid}]
It follows that the $i$-th (by length) principal semi-axis of $E^\prime$ is at most $\lambda^i$ for each $i \in [d]$. Moreover, at least $d - s$ of them must be equal. By direct computation, the maximal possible volume of $E^\prime$ is 
\[
{\volnf{E}}{\prod\limits_{j = 0}^{d - s - 1} \lambda^{-j}}.
\]
This bound is attained whenever $E^\prime$ is coaxial with $E$, and the lengths of some consecutive semi-axes of $E^\prime$ are equal to the corresponding $\lambda^i$ from $E$. However, this fixes the construction: $\ort{F}$ must be the linear span of $d - s$ consecutive principal semi-axes of $E$. Thus, $E^\prime$ is centered at the origin. The desired statements follow since the ratio of the lengths of one collinear semi-axis of $E$ and $E^\prime$ is $\lambda$.
\end{proof}

We note that the above argument can be extended to any convex body $K$ in the following way.

\begin{prp}
Let $K$ be a convex body in $\Red$, and let $A : \Red \to \Red$ be the diagonal linear operator with diagonal entries $\lambda, \lambda^2,  \ldots, \lambda^{d}$. Let $K'=A'K$ be a solution to \Href{Problem}{prob:genrotation} with some linear operator $A'$, for $K$ and $L=AK$, with any $d-s \geq 2$. Then, for any $\mu >0$ there is some value of $\lambda>0$ such that
\begin{enumerate}
\item[(a)]  for any $\mu' < \mu$, no homothetic copy $x+ \mu' K'$ of $K'$ contains $L$,
\item[(b)] $\frac{\volnf[d]{K'}}{\volnf[d]{L}} \leq \frac{1}{\mu}$.
\end{enumerate}
\end{prp}

\begin{proof}
Let $B_{\mathrm{is}}$ be a largest ball contained in $K$, and let $B_{\mathrm{cs}}$ be the smallest ball containing $K$. Assume that $B_{\mathrm{cs}}=y+ \nu B_{\mathrm{is}}$ for $\nu \geq 1$, $y \in \Red$.

Let $\bar{\mu} > 0$ be the smallest number such that a homothetic copy of $K'$ of ratio $\bar{\mu}$ contains $L$. Then there is a homothetic copy of $A'B_{\mathrm{cs}}$ of ratio $\bar{\mu}$ that contains $AB_{\mathrm{is}}$. Consequently, there is a homothetic copy of $A'B_{\mathrm{is}}$ of ratio $\bar{\mu} \nu^2$ that contains $AB_{\mathrm{cs}}$. Clearly, $A'B_{\mathrm{is}} \subseteq L\subseteq AB_{\mathrm{cs}}$. Thus, by the previous example, $\bar{\mu} \nu^2 \geq \lambda$, implying that $\bar{\mu} \geq \frac{\lambda}{\nu^2}$, proving (a).

To prove (b), we may apply a similar consideration.
\end{proof}

\section{Maximal volume ellipsoids of revolution}\label{sec:ellipsoids}
First, as immediate corollaries to \Hrefs{Theorems}{thm:genfixed}{thm:genrotation}, we obtain John type necessary conditions for an ellipsoid to be the largest volume ellipsoid of revolution in a given convex body.
Next, as the main goal of the section we prove \Href{Theorem}{thm:ellfixedproperties}, and discuss the tightness of the results. Our proof strategy is to substitute $K$ with a suitable contact polytope; see \Href{Section}{sec:containment}.

\subsection{Contact points on the sphere}\label{subsec:notation_ellipsoids}

In \Hrefs{Theorems}{thm:genfixed}{thm:genrotation}, we phrased our John type conditions in terms of the contact pairs of the two bodies. Since ellipsoids are affine images of the ball, we will, in this section, phrase our John type conditions for ellipsoids of revolution in terms of unit vectors that are pre-images of contact points under this affine transformation.

As an immediate corollary of \Href{Theorem}{thm:genfixed}, we obtain the following answer to \Href{Problem}{prob:ellfixed}.

\begin{cor}[Necessary condition for John's problem for ellipsoids with a fixed axis]\label{cor:johncondition_ell_fixed_F}
Let the origin-centered ellipsoid $E = A \ball{d}$ be a solution to \Href{Problem}{prob:ellfixed}, that is
\probellfixed{} for a convex body $K$ in $\Red$. 
Then there are unit vectors $u_1, \dots, u_m$, and positive weights $\alpha_1, \dots, \alpha_m$ such that
$(A u_1, \inv{A} u_1), \dots, (A u_m, \inv{A} u_m)$ are contact pairs of $E$ and $K$, and
\[
    P_F \parenth{\sum\limits_{i \in [m]} \alpha_i u_i\otimes u_i} P_F = P_F;
    \quad   \sum\limits_{i \in [m]} \alpha_i u_i =\origin ; 
    \quad   \sum\limits_{i \in [m]} \alpha_i = d.
\]
\end{cor}

\begin{proof}
    Observe that at a boundary point $Au\in\bd A\ball{d}$, the corresponding normal vector is $A^{-1}u$. 
The statement now follows from \Href{Theorem}{thm:genfixed}.
\end{proof}

Note that a standard argument using Carathéodory's theorem for the space $\rotpos$ shows that the number $m$ of points needed in \Href{Corollary}{cor:johncondition_ell_fixed_F} is at most $2d^2 + d + 1$.

As an immediate corollary of \Href{Theorem}{thm:genrotation}, we obtain the following answer to \Href{Problem}{prob:ellrotation}.

\begin{cor}[Necessary condition for John's problem for ellipsoids with any axis] \label{cor:johncondition_ell_rotation}
Assume that $E=A \ball{d}$ is a solution to \Href{Problem}{prob:ellrotation}, that is \probellrotation{} for a convex body $K$ in $\Red$. Then there exist $m \leq 2d^2 + d + 1$ unit vectors $u_1, \dots, u_m$ and positive weights $\alpha_1, \dots, \alpha_m$ such that the conclusions of \Href{Corollary}{cor:johncondition_ell_fixed_F} hold as well as 
\[
  A \parenth{ \sum\limits_{i \in [m]} \alpha_i u_i\otimes u_i } A^{-1} = 
    \transpose{\parenth{A^{-1}}} \parenth{ \sum\limits_{i \in [m]} \alpha_i u_i\otimes u_i } \transpose{A}.
 \]
\end{cor}

\begin{dfn}
We will call a position $K^\prime\in\apositions{K}$ with fixed axis $F$ of a convex body $K$ the \emph{John position with fixed axis $F$}, if the standard unit ball is a solution to John’s general problem with fixed axis $F$ for $K^\prime$ and $\ball{d}.$
\end{dfn}
In the case $F$ is the whole space, the John position with fixed $F$ corresponds to the standard John position of a convex body.

Using the matrix $A$ from \Href{Corollary}{cor:johncondition_ell_fixed_F}, we see that the position $\inv{A} K$ of $K$ is its John position with fixed axis $F$.
Throughout this section, we fix the corresponding unit vectors $u_1, \dots, u_m$ and positive weights $\alpha_1, \dots, \alpha_m$. Considering the vectors $P_F u_i$ as lying in $F$, we obtain:
\begin{equation}\label{eq:proj_contact_unit_vectors}
\sum\limits_{i \in [m]} \alpha_i P_F u_i\otimes P_F u_i = \id_F;
\quad \sum\limits_{i \in [m]} \alpha_i P_F u_i =\origin ;
\quad  \sum\limits_{i \in [m]} \alpha_i = d.
\end{equation}

As in problems concerning convex bodies in John position, we will study the corresponding ``outer contact polytopes'':

Denote
\[
K_{\text{J}} = \bigcap\limits_{i \in [m]}
\braces{x \in \Red \st \iprod{x}{u_i}\leq 1}.
\]

In the centrally symmetric case $K = -K$, define
\[
K_{\mathrm{J_{sym}}} = \bigcap\limits_{i \in [m]}
\braces{x \in \Red \st |{\iprod{x}{u_i}}|\leq 1}.
\]

\subsection{Inradius in \texorpdfstring{$\ort{F}$}{F orthogonal}}

\begin{lem}\label{lem:inradius_bound}
   Let $v_1,  \dots, v_m$ be vectors in $  \ball{n}.$
Assume there are positive weights $\beta_1, \dots, \beta_m$ such that 
\[
\sum\limits_{i \in [m]} \beta_i = 1, 
\quad 
\sum\limits_{i \in [m]} \beta_i v_i =\origin , 
\quad
\sum\limits_{i \in [m]} \beta_i \enorm{v_i}^2 = \theta
\text{ for some } \theta \in (0,1].
\]
Then the inradius of $\polarset{\parenth{\conv\{v_1, \dots, v_m\}}}$ is at most 
$\frac{1}{\theta}.$ 
\end{lem}
\begin{proof}
    Assume  $r \ball{n} + c \subset \polarset{\parenth{\conv\{v_1, \dots, v_m\}}}.$ This means that for every $i \in [m]$ with $v_i \neq 0,$ the point $c + r\frac{v_i}{\enorm{v_i}}$ is in $\polarset{\parenth{\conv\{v_1, \dots, v_m\}}}.$ Hence, 
    \[
    \iprod{v_i}{ c + r\frac{v_i}{\enorm{v_i}}} \leq 1 
        \quad \Longleftrightarrow \quad 
    \iprod{v_i}{c} + r\enorm{v_i}\leq 1  
    \quad \Longleftrightarrow \quad 
   r\enorm{v_i}  \leq 1 - \iprod{v_i}{c}.
    \]
 The latter inequality also holds if $v_i =\origin .$ Multiplying it by $\beta_{i}$ and summing, one gets
 \[
  r \parenth{\sum\limits_{i \in [m]} \beta_i \enorm{v_i}}
\leq \sum\limits_{i \in [m]} \beta_i -
 \iprod{\sum\limits_{i \in [m]} \beta_i v_i}{c}.
 \]
Since $\sum\limits_{i \in [m]} \beta_i v_i =\origin $ and $\sum\limits_{i \in [m]} \beta_i = 1,$ we get
\[
r \parenth{\sum\limits_{i \in [m]} \beta_i \enorm{v_i}}  \leq 1
\quad \Longleftrightarrow \quad 
r \leq \frac{1}{\sum\limits_{i \in [m]} \beta_i \enorm{v_i}}.
\]
Since $v_i \in \ball{n},$ we conclude that 
$\sum\limits_{i \in [m]} \beta_i \enorm{v_i} \geq 
\sum\limits_{i \in [m]} \beta_i {\enorm{v_i}^2} = \theta.$
The desired inequality follows.
\end{proof}

\begin{lem}
    \label{lem:inradius_in_F_perp}
    Under the notation of \Href{Subsection}{subsec:notation_ellipsoids}, with $s=\dim F$, the inradius of $K_{\text{J}} \cap \ort{F}$ satisfies
\[
  \inrad\! \parenth{ K \cap \ort{F}} \leq \frac{d}{d-s}\cdot\mathrm{radius}(E \cap \ort{F}).
\]

Additionally, in the case when $K=-K$, the inradius of $K_{\mathrm{J_{sym}}} \cap \ort{F}$ satisfies
\[
  \inrad\! \parenth{ K \cap \ort{F}} \leq  \sqrt{\frac{d}{d-s}}\cdot\mathrm{radius}(E \cap \ort{F}).
\]
\end{lem}

\begin{proof}
 The first identity in 
\eqref{eq:proj_contact_unit_vectors} yields $\sum\limits_{i \in [m]} \alpha_i \enorm{P_F u_i}^2 = s.$ Using here that $\sum\limits_{i \in [m]} \alpha_i = d$, we conclude $\sum\limits_{i \in [m]} \alpha_i \enorm{P_{\ort{F}}u_i}^2 = d - s.$

Since $\enorm{P_{\ort{F}}u_i} \leq \enorm{u_i} = 1$ for every $i \in [m]$, the vectors $P_{\ort{F}}u_1, \dots, P_{\ort{F}}u_m$ satisfy the assumptions of \Href{Lemma}{lem:inradius_bound} with  $\beta_1 = \frac{\alpha_1}{d}, \dots, \beta_m = \frac{\alpha_m}{d},$ and $\theta = \frac{d - s}{d}.$ The general bound follows.

In the centrally symmetric case, we know that the origin is the center of an inball. Hence,
the inradius cannot exceed the distance from the origin to a hyperplane 
$\braces{x \in \ort{F} \st \iprod{P_{\ort{F}}u_i}{x} = 1}$ in $\ort{F}$ for some non-zero $P_{\ort{F}}u_i.$ This distance is exactly $\frac{1}{\enorm{P_{\ort{F}}u_i}}.$ Identities \eqref{eq:answerGLMPfixedsubspace_decomposition} and \eqref{eq:answerGLMPfixedsubspace_trace} imply that $\enorm{P_{\ort{F}}u_j}^2 \geq \frac{d - s}{d}$ for some $j \in [m].$ The bound follows.
\end{proof}

\subsection{Volume ratio in \texorpdfstring{$F$}{F}}\label{sec:volratio}

The proofs of our volume bounds follow very closely the proof of Theorem~1.1 of \cite{AlonsoGutirrez2021}, where the authors study upper bounds on the volume of a section of a convex body in John position. The authors of \cite{AlonsoGutirrez2021} describe the corresponding section using the projections of the unit vectors that form the John decomposition in $\Re^d$. In the rest of the proof in \cite{AlonsoGutirrez2021}, only algebraic identities \eqref{eq:proj_contact_unit_vectors} and Ball's normalized version of the Brascamp--Lieb inequality \cite{ball1989volumes} are used.

\begin{lem}[Brascamp--Lieb inequality]\label{lem:balls_Brascamp_Lieb}
Let unit vectors $v_1,\ldots,v_m\in\Re^n$ and positive numbers $\beta_1,\ldots,\beta_m$ satisfy
\[
\sum\limits_{i\in[m]} \beta_i v_i\otimes v_i=\id_{\Re^n} .
\]
Then
\begin{equation}\label{eq:BrascampLieb}
\int_{\Re^n}\prod_{i\in[m]}
f_i(\iprod{x}{v_i})^{\beta_i} \di x 
\leq \prod_{i\in[m]}\left(\int_{\Re}f_i\right)^{\beta_i},
\end{equation}
holds for all integrable functions $f_1,\ldots,f_m\colon\Re\to[0,\infty).$
\end{lem}

\begin{lem}\label{lem:vr_sym_fixed_subsp_John}
Under the notation of \Href{Subsection}{subsec:notation_ellipsoids}, with $s=\dim F$, the inequality
\[
\parenth{\frac{\vol[s]{K_{\mathrm{J_{sym}}} \cap F}}{\volnf[s]{\cube^s}}}^\frac{1}{s}  \leq 
\sqrt{\frac{d}{s}}
\]
holds in the centrally symmetric case, $K=-K$.
\end{lem}

\begin{proof}
Let $I \subset [m]$ be the set of those indices $i\in[m]$ for which $P_Fu_i \neq 0$.  
Set
$v_i = \frac{P_F u_i}{\enorm{P_F u_i}}$ and
$\beta_i = {\alpha_i}{\enorm{P_F u_i}^2}$ for all $i \in I.$ 
Then
\[
\sum_{i\in I} \beta_i v_i\otimes v_i = \sum_{i\in[m]} \alpha_i P_F u_i\otimes P_F u_i = \id_F.
\]
Let $f_i$ denote the indicator function of the interval 
$
 \left[-\frac{1}{\enorm{P_F u_i}}, \frac{1}{\enorm{P_F u_i}}\right]
$ 
for each $i \in I$. 

We leave it to the reader to check that the integral on the left-hand side of \eqref{eq:BrascampLieb} is equal to the volume of $K_{\mathrm{J_{sym}}} \cap F$, as we identify $F$ with $\Re^n$ (where $n=s$), and that the product of the integrals on the right-hand side of \eqref{eq:BrascampLieb} is equal to 
$2^s \prod\limits_{i\in I}  \enorm{P_F u_i}^{-\alpha_i \enorm{P_F u_i}^2}.$ The latter quantity is at most $2^s \parenth{\frac{d}{s}}^{\frac{s}{2}}$ by Jensen's inequality, completing the proof of \Href{Lemma}{lem:vr_sym_fixed_subsp_John}.
\end{proof}

\begin{lem}\label{lem:vr_gen_fixed_subsp_John}
Under the notation of \Href{Subsection}{subsec:notation_ellipsoids}, with $s=\dim F$, the inequality
\[
\parenth{\frac{\vol[s]{K_{\text{J}} \cap F}}{\volnf[s]{\Delta_J^s}}}^\frac{1}{s}  \leq
\frac{1}{(s + 1)^{\frac{d - s}{2s(d + 1)}}} 
\sqrt{ \frac{d(d + 1)}{s(s + 1)} }
\]
holds.
\end{lem}

\begin{proof}
Similarly to Ball's argument \cite{ball1991volume}, we lift the picture to 
$\widetilde{F}=F\times \Re$ and will write 
$(a,t)$ for a vector in $\widetilde{F}$ with $a \in F$ and 
$t \in \Re.$

Let $I \subset [m]$ denote the subset of indices $i\in[m]$ for which $P_F u_i \neq 0$.
Denote $d^\prime = \sum\limits_{i \in I} \alpha_i;$ and set
\[
\widetilde{q}_i=
\parenth{-\sqrt{\frac{d^\prime}{d^\prime+1}} P_F u_i, 
\frac{1}{\sqrt{d^\prime+1}}}, \quad \text{and } \quad \widetilde{\alpha}_i=\frac{d^\prime+1}{d^\prime}\alpha_i.
\]
Set
$v_i = \frac{\widetilde{q}_i}{\enorm{\widetilde{q}_i}}$ and
$\beta_i = {\widetilde{\alpha}_i}{\enorm{\widetilde{q}_i}^2}$ for all $i \in I.$ 
By straightforward computation,
\[
\sum\limits_{i\in I} \beta_i v_i\otimes v_i = 
\sum\limits_{i\in I}
\widetilde{\alpha}_i\widetilde{q}_i\otimes\widetilde{q}_i = \id_{\widetilde{F}}.
\]
We will denote the function $ t \mapsto
e^{-\frac{t}{\enorm{\widetilde{q}_i}}} \indicator{[0, +\infty)}
$ by $f_i$
for each $i \in I, $ where $\indicator{S}$ stands for the indicator function of the set $S.$
Then, the product of the integrals on the right-hand side of \eqref{eq:BrascampLieb} is equal to 
$ \prod\limits_{i\in I}  \enorm{P_F u_i}^{\alpha_i \enorm{P_F u_i}^2}.$
The bound 
\begin{equation}
\label{eq:vr_product_bound}
   \prod\limits_{i\in I}  \enorm{P_F u_i}^{\alpha_i \enorm{P_F u_i}^2}\leq
\frac{(s+1)^{\frac{s+1}{2(d^\prime +1)}}}{(d^\prime + 1)^{\frac{d^\prime + 1}{2(d^\prime + 1)}}} 
\end{equation}

follows from the standard majorization argument for the convex function $t \mapsto t \ln t$ on $(0,1)$.
This bound was obtained in the proof of Theorem~1.1 in \cite{AlonsoGutirrez2021}.
For the sake of completeness, we provide the proof as a separate lemma in 
\Href{Appendix}{sec:major_lemma}.

We claim that the integral on the left-hand side of \eqref{eq:BrascampLieb} is
\[
\int_{\tilde{F}}\prod_{i\in I}
f_i(\iprod{x}{v_i})^{\beta_i} \di x =
\frac{s!}{(d^\prime+1)^{\frac{s+1}{2}} 
{d^\prime}^{\frac{s}{2}}}  \vol[s]{K_{\text{J}} \cap F}.
\]
Indeed, set $f(x)=\prod\limits_{i\in I}
f_i(\iprod{x}{v_i})^{\beta_i}$ for $x \in \tilde{F}$.
The support of $f$ is the cone
$\poscone{\braces{\parenth{\frac{K_{\text{J}} \cap F}{\sqrt{d^\prime}},1}}}$. 
Now, $f$ is constant $e^{- \tau{\sqrt{d^\prime+1}}}$ at the intersection of the hyperplane $H_\tau = \braces{(y,t) \in \widetilde{F} \st t = \tau}$ and this cone for any positive $\tau$. 
By monotonicity in $d^\prime$, the proof of \Href{Lemma}{lem:vr_gen_fixed_subsp_John} is complete.
\end{proof}
\begin{rem}
The authors of \cite{AlonsoGutirrez2021} used a different approach in the general case:
they first lift the ``contact vectors'' to $\Red \times \Re$ and then project the configuration onto $F \times \Re.$ This minor variation in the argument does not substantially affect the subsequent computations.
\end{rem}

\subsection{Proof of \Hreftitle{Theorem}{thm:ellfixedproperties}}

A transformation of the space by an $F$-matrix changes both the volume form on $F$ and the distance in $\ort{F}$ proportionally. Hence, under the notation of \Href{Subsection}{subsec:notation_ellipsoids}, it suffices to bound the volume of $\inv{A} K \cap F$ and the inradius of $\inv{A} K \cap \ort{F}$ to obtain the corresponding bound.

By construction, $\inv{A} K \subset K_{\text{J}}$ and $\inv{A} K \subset K_{\mathrm{J_{sym}}}$ in the centrally symmetric case $K = -K.$ The corresponding volume bounds follow from 
\Href{Lemma}{lem:vr_sym_fixed_subsp_John} and \Href{Lemma}{lem:vr_gen_fixed_subsp_John}; 
the inradius bound follows from \Href{Lemma}{lem:inradius_in_F_perp}.  

The inclusion $E \cap F \subset K \cap F \subset d \parenth{E \cap F}$ follows from \Href{Theorem}{thm:GLMP_inclusion_in_F} with $z=\origin$ and $P_F E = E \cap F \subset K \cap F.$

Let us now establish the inclusion
\[
E \cap F \subset K \cap F \subset \sqrt{d} \parenth{E \cap F}
\]
in the centrally symmetric case $K = -K$.  
Fix $x \in K_{\mathrm{J_{sym}}} \cap F.$ One has
\[
\iprod{x}{x} \stackrel{\eqref{eq:proj_contact_unit_vectors}}{=}
\sum\limits_{i \in [m]} \alpha_i \iprod{x}{P_F u_i} \iprod{P_F u_i}{x}
= \sum\limits_{i \in [m]} \alpha_i
\iprod{P_F x}{u_i} \iprod{u_i}{P_F x} \stackrel{x \in F}{=}
\]
\[
\sum\limits_{i \in [m]} \alpha_i
\iprod{x}{u_i} \iprod{u_i}{x} =
\sum\limits_{i \in [m]} \alpha_i \iprod{x}{u_i}^2 \leq \sum\limits_{i \in [m]} \alpha_i = d.
\]
The desired bound follows. The proof of \Href{Theorem}{thm:ellfixedproperties} is complete.

\subsection{Tightness of the results}\label{sec:volratiotight}
The inclusions of \Href{Theorem}{thm:ellfixedproperties} are optimal since they are already optimal for the John ellipsoid as shown by the simplex and the cube. The volume bounds of \Href{Lemma}{lem:vr_sym_fixed_subsp_John} and 
\Href{Lemma}{lem:vr_gen_fixed_subsp_John} coincide with the corresponding bounds for an 
$s$-dimensional section of a $d$-dimensional convex body in John position (see Theorem 1.1 of \cite{AlonsoGutirrez2021}). In the centrally symmetric case $K = -K,$ we conjecture the following:
\begin{conj}\label{conj:volrat_John_rev_cube_conjecture}
  Let an origin-centered ellipsoid $E$ be a solution to \Href{Problem}{prob:ellfixed} for a centrally symmetric convex body $K$ and a linear subspace $F$ in $\Red$ of dimension $s\in[d]$. 
Then 
\[
\frac{\vol[s]{K \cap F}}{\vol[s]{E \cap F}} \leq C_{\cube}(d,s)
\frac{\volnf[s]{\cube^s}}{\volnf[s]{\ball{s}}},
\]
where the constant $C_{\cube}(d,s)$ is given by
\[
C^2_{\cube}(d,s)=\left\lceil\frac{d}{s}\right\rceil^{d-s\lfloor d/s\rfloor}\left\lfloor\frac{d}{s}\right\rfloor^{s-(d-s\lfloor d/s\rfloor)}.
\]
\end{conj}

The constant $C_{\cube}(d,s)$ in \Href{Conjecture}{conj:volrat_John_rev_cube_conjecture}
is attained on certain sections of the cube $\cube^d=[-1,1]^d$ by $s$-dimensional subspaces
(see Lemma 2.2 of \cite{ivanov2020volumel}). Peculiarly, in the case when $K$ is centrally symmetric and $s$ divides $d,$ 
both our volume ratio bound in $F$ and the inradius bound in $\ort{F}$ are simultaneously tight: 
$K$ is the cube $[-1,1]^d;$ $E$ is the John ellipsoid $\ball{d}$ of $K;$
$F$ is the solution set to the system of linear equations:
\[
\iprod{x}{e_{\frac{d}{s}(i-1) +1}} = \ldots =
\iprod{x}{e_{\frac{d}{s}i}},  
\quad \text{for all} \quad  i \in [s].
\]

Moreover, the bound $\sqrt{\frac{d}{d-s}}$ on the inradius in $\ort{F}$ 
is attained on certain sections of the cube $\cube^d.$ Indeed, denoting the standard basis of $\Red$ by $u_1, \dots, u_d,$ the estimate $\sqrt{\frac{d}{d-s}}$ is attained whenever 
$\enorm{P_{\ort{F}} u_1}^2= \dots = \enorm{P_{\ort{F}} u_d}^2 = \frac{d-s}{d}.$ The existence of subspaces with this property was proven in Lemma 2.4 of \cite{ivanov2020volume}, in which the set of all possible lengths of the projections was described.

The general case of a not necessarily symmetric body is more intricate. 
As shown in \cite{dirksen2017sections}, the bound in \Href{Lemma}{lem:vr_gen_fixed_subsp_John} is asymptotically tight for a regular simplex $\Delta_J^d$ in the John position. Consequently, it is asymptotically tight in our case. 
However, we believe that there is an intrinsic problem in obtaining the best possible bound using the Brascamp--Lieb inequality directly: if one chooses the parameters to maximize the functional in \Href{Lemma}{lem:vr_convex_func_max}, then the Brascamp--Lieb inequality is not optimal for the corresponding functions. That is, there are geometric obstructions to achieving the equality case in the Brascamp--Lieb inequality.

We also have an example showing that the bound 
\Href{Theorem}{thm:ellfixedproperties} \ref{item:inradius} on the inradius is the best possible. It is technical, and we state and prove it in \Href{Theorem}{thm:tightness_inrad_bound}.
Nevertheless, the corresponding vectors in $\Re^{d-s}$ can be lifted to unit vectors satisfying the algebraic equations of \Href{Corollary}{cor:johncondition_ell_fixed_F}.
A more detailed analysis of \Href{Lemma}{lem:inradius_bound} shows that the bound can be achieved, if certain vectors are close to zero and the others lie near the unit sphere, so it cannot be achieved via the projections of the unit vectors from the standard John condition for $F = \Red$ (see \cite{Ball2009}).

Finally, \cite{ivanov2022functional} by M. Naszódi and G. Ivanov may be interpreted as studying \probgenfixed{} $F$ for $\ball{n+1}$ and $K$ being a certain not necessarily convex set in $\Re^{n+1}$ associated with a log-concave function on $\Re^n$, and $F$ being $\Re^n.$ 
Moreover, the algebraic conditions for the “contact points” in that case coincide with those of \Href{Corollary}{cor:johncondition_ell_fixed_F}; and the more general case considered in \cite{ivanovIMRN} fits within the framework of \Href{Theorem}{thm:genfixed}
for appropriately chosen parameter sets. The inclusions for ellipsoids and ``contact-type'' polytopes have been studied previously in the functional setting as well; see, for example, \cite[Lemma 4.1]{ivanov2024johninclusionlogconcavefunctions}.

\section{Uniqueness of the solution}\label{sec:uniqueness}

We note that in general, the solution of {\probellrotation} is not unique: consider the convex body $K=E+\left[-v,v\right]$, where $E$ is an $F$-ellipsoid for some $s$-dimensional subspace $F$ and $v\in \ort{F}$. On the other hand, we have the following theorem. 

For a linear subspace $F$ and a subset $K$ of $\Red$, we denote by $\mathcal{S}_F(K)$ the set of those elements of $\apositions{K}$, for which the corresponding linear transformation is positive definite.

\begin{thm}\label{thm:translate}
For any convex bodies $K,L$ in $\Red$, and for any elements $K_1, K_2 \in \mathcal{S}_F(K)$ of maximal volume satisfying $K_1, K_2 \subseteq L$, $K_1$ and $K_2$ are translates of each other.
\end{thm}

\begin{proof}
Consider the family of pairs $(S, x )$, where $S$ is a positive semidefinite $F$-operator, $x \in \Red$ satisfying $x+S(K) \subseteq L$. We regard $(S,x)$ as a point of $\Re^{d^2+d}$. It is easy to see that the fa\-mi\-ly of pairs $(S,x)$ with the above condition is a compact, convex subset of $\Re^{d^2+d}$, and thus, $\vol[d]{x+S(K)}=\det(S) \volnf[d]{K}$ attains its maximum on it.
Note also that if $x+S(K)$ maximizes $\vol[d]{x+S(K)}$ on this set, then $S$ is positive definite, as otherwise $\vol[d]{x+S(K)} = 0$.

Consider some positive definite $F$-operators $S_1, S_2$ and vectors $x_1,x_2 \in \Red$ such that for $i=1,2$, $K_i=x_i+S_i(K)$ maximize volume.
By convexity, we have that $\bar{K}=\frac{x_1+x_2}{2} + \frac{1}{2} (S_1+S_2)(K) \subseteq L$. On the other hand, it is well-known that for any two positive definite matrices $A,B$,
\[
\det \frac{A+B}{2} \geq \sqrt{\det A \cdot \det B}
\]
with equality if and only if $A=B$ (see e.g. \cite[Corollary 7.6.8]{horn2012matrix}).
Since $\frac{1}{2} (S_1+S_2)$ is also a positive definite $F$-operator, it follows from the maximality of volume that $S_1=S_2$, implying that $K_1$ and $K_2$ are translates of each other.
\end{proof}

\begin{rem}
We note that $\mathcal{S}_F(\ball{d}) = \apositions{\ball{d}}$ for every linear subspace $F$, by the polar decomposition of matrices.
\end{rem}

\begin{rem}
It is not true that if $K_1, K_2 \in \mathcal{S}_F(K)$ are of maximal volume satisfying $K_1, K_2 \subseteq L$, then the translation vector moving $K_1$ to $K_2$ is necessarily parallel to $\ort{F}$, as shown by the following example. Set $d=2$ and $K=\ball{2}$. We define $L$ as an $\origin$-symmetric rectangle containing $\ball{2}$ such that the sides of $L$ are parallel to the lines $y = \pm x$, and the only points of $\bd{L}$ on $\ball{2}$ are the points $\pm \left( \frac{1}{\sqrt{2}}, \frac{1}{\sqrt{2}}\right)$. Let $F$ be the $y$-axis. Then an elementary computation shows that the only $\origin$-symmetric $F$-ellipse inscribed in $L$ is $\ball{2}$, and thus, this is of maximal area. On the other hand, moving this ellipse along the line $y=-x$ yields other maximal area $F$-ellipses in $L$.
\end{rem}

\section{\texorpdfstring{$F$}{F}-revolution L\"owner ellipsoids}\label{sec:lowner}

A dual construction to the largest volume ellipsoid contained in a convex body
is the smallest volume ellipsoid containing the body, usually called the \emph{L\"owner ellipsoid}. Notably, the necessary and sufficient
conditions for the Euclidean unit ball to be the L\"owner ellipsoid coincide with the
conditions in John’s characterization of the largest volume ellipsoid.
The extension of the notion of the L\"owner ellipsoid is straightforward -- swap the bodies in \probgenfixed{} $F$, that is, consider \probgenfixed{} $F$ for $K$ and $\ball{d}.$

\begin{dfn}
We will call a position $K^\prime\in\apositions{K}$ with fixed axis $F$ of a convex body $K$ the \emph{L\"owner position with fixed axis $F$} if $K^\prime$ is a solution to \probgenfixed{} $F$ for $K$ and $\ball{d}.$
\end{dfn}

The following result is a dual version of \Href{Theorem}{thm:ellfixedproperties}:
\begin{thm}\label{thm:ellLownerfixedproperties}
  Let a convex body $K$ be in the L\"owner position with fixed axis $F$ of dimension $s\in\{0,1,\dots,d\}$.
    Then 
\begin{enumerate}[label=({\arabic*})]
\item\label{item:outervolbound}
the following outer volume bound holds:
\[
    \parenth{\frac{\volnf[s]{ P_FK}}{\volnf[s]{\Delta_L^s}}}^\frac{1}{s} \geq
\sqrt{\frac{s}{d}},
\]
where $\Delta_L^s$ stands for the regular $s$-dimensional simplex inscribed in $\ball{s}$;
\item\label{item:lownerinclusion}
the following inclusion holds:
\[
P_F K \subset P_F \ball{d} \subset d \cdot P_F K;
\]
    \item\label{item:circumradius}
    the circumradius of $P_{\ort{F}} K$ is at least
$
 \sqrt{\frac{d-s}{d}}.
$
\end{enumerate}

If, additionally, $K$ is centrally symmetric, then
\begin{enumerate}[label=({S\arabic*})]
\item\label{item:outervolboundsym}
the following outer volume bound holds:
\[
    \parenth{\frac{\volnf[s]{P_F K}}{\volnf[s]{\crosp^s}}}^\frac{1}{s} \geq
\sqrt{\frac{s}{d}},
\]
where $\crosp^s$ stands for the standard cross-polytope;
\item\label{item:inclusionlownersym}
the following inclusion holds:
\[
P_F K \subset P_F \ball{d} \subset \sqrt{d} \cdot P_F K.
\]
\end{enumerate}
\end{thm}

We will prove this theorem in the current section. The proof mostly consists of standard arguments dual to those used in the proof of \Href{Theorem}{thm:ellfixedproperties}. The only surprise is that the bound \ref{item:circumradius} on the circumradius does not require $K$ to be symmetric!

\subsection{Notation}\label{subsec:notation_lowner_ellipsoids}

As a direct corollary of \Href{Theorem}{thm:genfixed}, we obtain:
\begin{cor}\label{cor:lownercondition_ell_fixed_F}
  Let a convex body $K$ be in L\"owner position with fixed axis $F.$
  Then there exist $m \leq 2d^2 + d + 1$, unit vectors $u_1, \dots, u_m$, and positive weights $\alpha_1, \dots, \alpha_m$ such that 
  $( u_1, u_1), \dots, ( u_m,  u_m)$ are contact pairs of $K$ and $\ball{d}$, and
\[
    P_F \parenth{\sum\limits_{i \in [m]} \alpha_i u_i\otimes u_i} P_F = P_F;
    \quad   \sum\limits_{i \in [m]} \alpha_i u_i =\origin ; 
    \quad   \sum\limits_{i \in [m]} \alpha_i = d.
\]
\end{cor}

In the case where $F$ is the whole space, the L\"owner position with fixed axis $F$ coincides with the standard L\"owner position of a convex body. 

Throughout this section, we fix the corresponding unit vectors $u_1, \dots, u_m$ and positive weights $\alpha_1, \dots, \alpha_m$. Considering the vectors $P_F u_i$ as lying in $F$, we observe that they satisfy equations \eqref{eq:proj_contact_unit_vectors}.

As in problems concerning convex bodies in the John position, we will study the corresponding ``contact polytopes'':

Define
\[
K_{\text{L}} =  \conv\braces{u_1, \dots, u_m}.
\]

In the centrally symmetric case $K = -K$, define
\[
K_{\mathrm{L_{sym}}} = \conv\braces{\pm u_1, \dots, \pm u_m}.
\]

\subsection{Circumradius in \texorpdfstring{$\ort{F}$}{F orthogonal}}

\begin{lem}\label{lem:circumradius_bound}
   Let $v_1,  \dots, v_m$ be vectors in $\Re^n.$
   Assume that there are positive weights $\beta_1, \dots, \beta_m$ such that 
\[
\sum\limits_{i \in [m]} \beta_i = 1, 
\quad 
\sum\limits_{i \in [m]} \beta_i v_i =\origin , 
\quad
\sum\limits_{i \in [m]} \beta_i \enorm{v_i}^2 = \theta.
\]
Then the circumradius of $\conv\{v_1, \dots, v_m\}$ is at least 
$\sqrt{\theta}.$ 
\end{lem}
\begin{proof}
    Assume $r \ball{n} + c \supset \conv\{v_1, \dots, v_m\}.$ This means that 
    $\enorm{c - v_i}^2 \leq r^2$ for every $i \in [m].$ Multiplying both sides by $\beta_{i}$ and summing over $i$, we get:
 \[
  r^2 \parenth{\sum\limits_{i \in [m]} \beta_i } \geq 
 \sum\limits_{i \in [m]} \beta_i \enorm{c - v_i}^2 =
   \enorm{c}^2 \parenth{\sum\limits_{i \in [m]} \beta_i} -
   2 \iprod{c}{ \sum\limits_{i \in [m]} \beta_i v_i} + \sum\limits_{i \in [m]} \beta_i \enorm{v_i}^2.
 \]
Since $\sum\limits_{i \in [m]} \beta_i v_i =\origin $ and $\sum\limits_{i \in [m]} \beta_i =1$, we obtain
\[
r^2 \geq \sum\limits_{i \in [m]} \beta_i \enorm{v_i}^2 + \enorm{c}^2 \geq \sum\limits_{i \in [m]} \beta_i \enorm{v_i}^2 = \theta.
\]
The desired inequality follows.
\end{proof}

Combining \eqref{eq:proj_contact_unit_vectors} with the above lemma with $v_i=P_{F^\perp}u_i$ and $\beta_i=\alpha_i/d$, and by identifying $F^\perp$ with $\Re^n$ (where $n=d-s$), we conclude the following:
\begin{lem}
    \label{lem:circumradius_in_F_perp}
    Under the notation of \Href{Subsection}{subsec:notation_lowner_ellipsoids}, the circumradius of $P_{\ort{F}} K_{L}$ is at least
$
 \sqrt{\frac{d-s}{d}}.
$ 
\end{lem}

\subsection{Outer volume ratio in \texorpdfstring{$F$}{F}}
The proofs of our outer volume bounds follow from Theorem 1.2 of \cite{AlonsoGutirrez2021}, where the authors study lower bounds on the volume of projections of convex bodies in the L\"owner position. 
Starting with the breakthrough work \cite{BartheThesis} of Barthe, the main tool for obtaining outer volume ratio bounds is the following reverse form of the Brascamp--Lieb inequality:

\begin{lem}[Reverse Brascamp--Lieb inequality]
\label{lem:reverse_BL}
Let unit vectors $v_1,\ldots,v_m \in \Re^n$ and positive numbers $\beta_1,\ldots,\beta_m$ satisfy
$
\sum\limits_{i\in[m]} \beta_i v_i\otimes v_i = \id_{\Re^n}$, and
$g_1,\ldots,g_m\colon\Re\to[0,\infty)$ be integrable functions.
Define
\[
G(x) = \sup\braces{
\prod\limits_{i\in [m]}g_i(\theta_i)^{\beta_i} \st \theta_i\in\Re \text{ and }
x = \sum\limits_{i\in[m]}\beta_i \theta_i v_i}.
\]
Then,
\begin{equation}\label{eq:Barthe}
\int_{\ReN} G 
\geq 
\prod\limits_{i\in[m]}\parenth{\int g_i}^{\beta_i}.    
\end{equation}
\end{lem}

To simplify the reading, we provide a sketch of the proof of the following lemma, omitting technical details.

\begin{lem}\label{lem:ovr_sym_fixed_subsp_John}
Under the notation of \Href{Subsection}{subsec:notation_lowner_ellipsoids}, the inequality
\[
\parenth{\frac{\volnf[s]{P_F K_{\mathrm{L_{sym}}}}}
{\volnf[s]{\crosp^s}}}^\frac{1}{s}  \geq 
\sqrt{\frac{s}{d}}
\]
holds.
\end{lem}

\begin{proof}
Let $I \subset [m]$ be the set of those indices $i\in[m]$ for which $P_F u_i \neq 0.$  
Set
$v_i = \frac{P_F u_i}{\enorm{P_F u_i}}$ and
$\beta_i = {\alpha_i}{\enorm{P_F u_i}^2}$ for all $i \in I.$ 
Then
\[
\sum_{i \in I} \beta_i v_i \otimes v_i = \sum_{i \in [m]} \alpha_i P_F u_i \otimes P_F u_i = \id_F.
\]
Define 
\[
g_i(t) = e^{-\frac{\enorm{t}}{\enorm{P_F u_i}}}
\]
for each $i \in I$. 

We leave it to the reader to verify that the product of the integrals on the right-hand side of \eqref{eq:Barthe} equals 
$2^s \prod\limits_{i \in I} \enorm{P_F u_i}^{\alpha_i \enorm{P_F u_i}^2}$. This is at least $2^s \parenth{\frac{d}{s}}^{\frac{s}{2}}$ by Jensen's inequality. 

On the other hand, the corresponding function $G$ equals $e^{-\norm{\,\cdot\,}_{K_{\mathrm{L_{sym}}}}},$ where $\norm{\,\cdot\,}_{K_{\mathrm{L_{sym}}}}$ is the gauge function of $K_{\mathrm{L_{sym}}}:$
\[
\norm{x}_{K_{\mathrm{L_{sym}}}} = \inf\braces{t > 0 \st x \in t K_{\mathrm{L_{sym}}}}.
\]
Hence, by standard computation, the integral of $G$ is 
\[
\int_{\Re^s} G = s! \volnf[s]{P_F K_{\mathrm{L_{sym}}}}.
\]
The desired bound follows, since $\volnf[s]{\crosp^s} = \frac{2^s}{s!}.$
\end{proof}

\begin{lem}\label{lem:ovr_fixed_subsp_John}
Under the notation of \Href{Subsection}{subsec:notation_lowner_ellipsoids}, the inequality
\[
\parenth{\frac{\volnf[s]{P_F K_{L}}}
{\volnf[s]{\Delta_L^s}}}^\frac{1}{s}  \geq 
\sqrt{\frac{s}{d}}
\]
holds.
\end{lem}

\begin{proof}
Following Barthe's argument \cite{BartheThesis}, we lift the picture to 
$\widetilde{F} = F \times \Re$, and write 
$(a,t)$ for a vector in $\widetilde{F}$ with $a \in F$ and 
$t \in \Re.$
Let $I \subset [m]$ be the index set of those $i\in[m]$ such that $P_F u_i \neq 0.$  
Let $d^\prime = \sum\limits_{i \in I} \alpha_i$, and define
\[
\widetilde{q}_i = \parenth{-\sqrt{\frac{d^\prime}{d^\prime + 1}} P_F u_i, 
\frac{1}{\sqrt{d^\prime + 1}}}, \quad 
\widetilde{\alpha}_i = \frac{d^\prime + 1}{d^\prime} \alpha_i.
\]
Set
$v_i = \frac{\widetilde{q}_i}{\enorm{\widetilde{q}_i}}$ and
$\beta_i = \widetilde{\alpha}_i \enorm{\widetilde{q}_i}^2$ for all $i \in I.$ 
Then,
\[
\sum\limits_{i \in I} \beta_i v_i \otimes v_i =
\sum\limits_{i \in I}
\widetilde{\alpha}_i \widetilde{q}_i \otimes \widetilde{q}_i =
\id_{\widetilde{F}}.
\]

Define a gauge-like function $C$ on $\widetilde{F}$ by
\[
C(x) = \inf \braces{
\sum_{i \in I} \lambda_i \colon \lambda_i \geq 0 \text{ and } x = \sum_{i \in I} \lambda_i \widetilde{q}_i },
\]
where we set $C(x) = +\infty$, if such a decomposition does not exist.
Let $G(x) = e^{-C(x)}$.

By convexity, the support of $G$ is the cone
$\poscone{\parenth{P_F K_{\text{L}}, \frac{1}{\sqrt{d^\prime}}}}$; the function $G$ is constant and equal to $e^{-\tau \sqrt{d^\prime + 1}}$ on the intersection of this cone with the hyperplane $H_\tau = \braces{(x,t) \in \widetilde{F} \st t = \tau}$ for any $\tau > 0$. Thus, 
\[
\int_{\widetilde{F}} G  = 
\volnf[s]{P_F K_{\text{L}}}  \cdot \int_0^\infty e^{-r \sqrt{d^\prime + 1}} 
(r \sqrt{d^\prime})^s \di r =
\frac{s!(d^\prime)^{s/2}}{(d^\prime + 1)^{(s+1)/2}} \volnf[s]{P_F K_{\text{L}}}.
\]

Now define
\[
g_i(t) = e^{-t/\enorm{\widetilde{q}_i}} \cdot \indicator{[0, +\infty)}(t)
\]
for each $i \in I$. Then, for any $x$ in the support of $G,$
\[
\sup\braces{
\prod\limits_{i \in I} g_i(\theta_i)^{\beta_i} \st 
x = \sum\limits_{i \in I} \beta_i \theta_i v_i}
\stackrel{\lambda_i = \widetilde{\alpha}_i \enorm{\widetilde{q}_i} \theta_i}{=}
e^{-\inf\braces{
\sum\limits_{i \in I} \lambda_i \st 
x = \sum\limits_{i \in I} \lambda_i \widetilde{q}_i}} = G(x).
\]

Thus, by the reverse Brascamp--Lieb inequality (\Href{Lemma}{lem:reverse_BL}),
\[
\int_{\Red} G 
\geq 
\prod\limits_{i \in I} \enorm{\widetilde{q}_i}^{\widetilde{\alpha}_i \enorm{\widetilde{q}_i}^2}.
\]
Applying Jensen’s inequality with 
$\sum\limits_{i \in I} \widetilde{\alpha}_i \enorm{\widetilde{q}_i}^2 = s + 1$ 
and $\sum\limits_{i \in I} \widetilde{\alpha}_i = d^\prime +1$, we obtain
\[
\prod\limits_{i \in I} \enorm{\widetilde{q}_i}^{\widetilde{\alpha}_i \enorm{\widetilde{q}_i}^2} 
\geq \parenth{\frac{s + 1}{d^\prime + 1}}^{(s + 1)/2}.
\]
Substituting into the expression for $\int_{\widetilde{F}} G$, we find
\[
\volnf[s]{P_F K_{\text{L}}} \geq \parenth{\frac{s + 1}{d^\prime + 1}}^{(s + 1)/2} 
\frac{(d^\prime + 1)^{(s + 1)/2}}{s! (d^\prime)^{s/2}} = 
\parenth{\frac{s}{d^\prime}}^{s/2} \volnf[s]{\Delta_L^s}.
\]
The desired bound follows since $d^\prime \leq d.$
\end{proof}

\begin{rem}
The bound is optimal whenever $s + 1$ divides $d + 1$. It is attained when 
$K = \Delta_L^d$ and $F$ is the subspace spanned by the centroids of $s+1$ pairwise disjoint $\parenth{\frac{d+1}{s+1} - 1}$-dimensional faces of $\Delta_L^d.$
\end{rem}

\subsection{Proof of \Hreftitle{Theorem}{thm:ellLownerfixedproperties}}
\begin{proof}[Proof of \Href{Theorem}{thm:ellLownerfixedproperties}]
We use the notation introduced in \Hrefs{Subsections}{subsec:notation_ellipsoids}{subsec:notation_lowner_ellipsoids}.
By duality, $K_{\text{L}} = \polarset{\parenth{K_{\text{J}}}}$ and 
$K_{\mathrm{L_{sym}}} = \polarset{ \parenth{K_{\mathrm{J_{sym}}}}}.$ Hence, the inclusions follow from the corresponding inclusions of \Href{Theorem}{thm:ellfixedproperties}.
The bound on the circumradius was obtained in \Href{Lemma}{lem:circumradius_in_F_perp}. The volume bounds were proven in \Href{Lemma}{lem:ovr_fixed_subsp_John} in the general case and in \Href{Lemma}{lem:ovr_sym_fixed_subsp_John} in the centrally symmetric case $K = -K.$

\end{proof}

\appendix
\section{Tightness of the inradius bound}\label{sec:tightness_inrad_bound}

The goal of this section is to state and prove \Href{Theorem}{thm:tightness_inrad_bound} according to which the bound given in \ref{item:inradius} of \Href{Theorem}{thm:ellfixedproperties} on the inradius is the best possible.

Set
\[
a = \frac{1}{\sqrt{n}}(1, \dots, 1) \in \Re^n, \quad n \geq 2.
\]
We will define vectors $v_1, \dots, v_{n+1}$ that form a directionally dilated regular simplex along the direction of $a$.

For each $i\in [n]$, define the vector $p_i \in \Re^n$ whose $j$-th coordinate is given by
\[
p_i[j] = 
\begin{cases}
\frac{n - 1}{n}, & \text{if } j = i, \\
-\frac{1}{n}, & \text{if } j \neq i.
\end{cases}
\]
Then
$
\enorm{p_i}^2  = \frac{n - 1}{n}.
$
Define 
$
 q_i = \frac{p_i}{\enorm{p_i}} \in \Re^n.
$
Note that for each $i \in [n]$, the vectors $p_i$ and $q_i$ belong to $\ort{a}$, and that $\iprod{q_i}{q_j}=-\frac{1}{n-1}$ for any $1\leq i<j\leq n$.

For any $t\in(0,1)$, define the vectors
\[
v_i = t q_i + \sqrt{1 - t^2} \cdot a, \quad \text{for } i \in [n]
\]
Now, fix a positive constant $\gamma$ satisfying
\begin{equation}
    \label{eq:example_gamma_condition}
    n \gamma < 1.
\end{equation}

We observe that
\[
\sum_{i \in [n]} p_i =\origin  = \sum_{i \in [n]} q_i,
\]
\[
\sum_{i \in [n]} v_i = t \sum_{i \in [n]}q_i + n \sqrt{1 - t^2} \cdot a = n \sqrt{1 - t^2} \cdot a .
\]

Define $v_{n+1}$ by the identity
\begin{equation}\label{eq:gammasum}
 \gamma \sum_{i \in [n]} v_i + (1 - n \gamma) v_{n+1} =\origin.
\end{equation}
So, 
\[
v_{n+1} = - \frac{\gamma n}{ 1 - \gamma n}\sqrt{1 - t^2} \cdot a.
\]

Our second restriction is
\begin{equation}
    \label{eq:example_last_v_ineq}
    \frac{\gamma n}{ 1 - \gamma n}\sqrt{1 - t^2} < 1.
\end{equation}

Denote $\Delta = \conv\braces{v_1, \dots, v_{n+1}}$.
Let us estimate the inradius of the polar $\polarset{\Delta}$ of $\Delta$.
A direct computation yields the following:
\begin{lem}
\label{lem:example_comp_dual_simplex}
    The vertices of the simplex $\polarset{\Delta}$ are 
    \[
x_{n+1} = \frac{1}{ \sqrt{1 - t^2}} a,
\]
and 
\[
x_i = -\frac{1}{n \gamma}\frac{(n - 1)}{t} q_i - \frac{1 - n \gamma}{n \gamma} \cdot \frac{1}{\sqrt{1 - t^2}} a, \quad \text{for } i \in [n].
\]
\end{lem}

\begin{lem}
\label{lem:example_sectoin_third_point}
The section of $\polarset{\Delta}$ by the plane $L$ spanned by the vectors $x_1$ and $x_{n+1}$ is the triangle $T$ with vertices $x_1$, $x_{n+1}$, and $y$, given by
\[
y = \frac{1}{n \gamma} \cdot \frac{q_1}{t} - \frac{1 - n\gamma}{n \gamma} \cdot \frac{a}{\sqrt{1 - t^2}}.
\]
\end{lem}

\begin{proof}
Clearly, 
\[
L = \operatorname{span}\braces{v_1, v_{n+1}} = \operatorname{span}\braces{q_1, a},
\]
and $q_1$ and $a$ form an orthonormal basis of $L$.

By duality, the polar in $L$ of $\polarset{\Delta} \cap L$ is the orthogonal projection of $\Delta$ onto $L$.
We compute this projection. First, $v_1$ and $v_{n+1}$ already lie in $L$.
Second, for each $i \in [n] \setminus \{1\}$, 
\[
P_L v_i = \iprod{v_i}{q_1} q_1 + \iprod{v_i}{a} a = t \iprod{q_1}{q_i} q_1 + \sqrt{1 - t^2} \cdot a = 
-\frac{t}{n - 1} q_1 + \sqrt{1 - t^2} \cdot a,
\]
where $P_L$ denotes the orthogonal projection onto $L$.
In particular, all $v_2, \dots, v_n$ project onto the same point.

Thus, $P_L \Delta$ is the triangle with vertices $v_1$, $v_{n+1}$, and $P_L v_2$.
Hence, $\polarset{\Delta} \cap L$ is a triangle, say $T$. Let us compute its vertices.
Since $x_1$ and $x_{n+1}$ lie in $L$, they are two of the vertices of $T$.
The remaining vertex $y$ can be found using duality in $L$:
\[
y = \frac{1}{n \gamma} \cdot \frac{q_1}{t} - \frac{1 - n\gamma}{n \gamma} \cdot \frac{a}{\sqrt{1 - t^2}}.
\]
Indeed, 
\[
\iprod{y}{v_{n+1}} = 
\frac{1 - n\gamma}{n \gamma} \cdot \frac{1}{\sqrt{1 - t^2}} \iprod{a}{v_{n+1}} = 1,
\]
and
\[
\iprod{y}{v_1} =
\frac{1}{n \gamma} \cdot \iprod{\frac{q_1}{t}}{t q_1} 
- \frac{1 - n\gamma}{n \gamma} \cdot \iprod{\frac{a}{\sqrt{1 - t^2}}}{\sqrt{1 - t^2} a}= 
\frac{1}{n \gamma} - \frac{1 - n\gamma}{n \gamma} = 1.
\]
The lemma follows.
\end{proof}

\begin{lem}
Let $h$ be the foot of the altitude from the vertex $x_{n+1}$ of the triangle $T$, and let $z$ be the intersection point of the altitude $x_{n+1}h$ and the angle bisector at vertex $y$. Then the inradius $R(t)$ of $\polarset{\Delta}$ equals $\enorm{h - z}$. Moreover,
\[
R(t) = 
\frac{1}{n \gamma} \cdot \frac{1}{1 + \sqrt{1 - t^2}}.
\]
\end{lem}

\begin{proof}
Consider a regular simplex $K$ centered at the origin with a facet $G$ parallel to $\ort{a}$. Let $g$ be the vertex opposite to $G$, and let $v$ be any vertex of $G$. By symmetry, the origin is the incenter of $K$. In particular, the altitude from $v$ passes through the origin and lies in the plane $L = \operatorname{span}\{g, v\}$.

After applying a directional dilation along the direction of $a$ that maps $K$ to a simplex $K'$, the incenter of $K'$ remains on the line with direction vector $a$. Consequently, the incenter and its orthogonal projection onto the facet opposite to $v$ remain in the plane $L$.

In our setting, this shows that the incenter of $\polarset{\Delta}$ lies in $T$ and is the intersection point $z$ of the altitude $x_{n+1}h$ and the angle bisector from vertex $y$.

Let us now compute $R(t)$. We use identities from \Href{Lemma}{lem:example_comp_dual_simplex} and \Href{Lemma}{lem:example_sectoin_third_point} for $x_1$, $x_{n+1}$, and $y$.

By construction, the line passing through $y$ and $x_1$ is orthogonal to $a$, thus
\[
h =  - \frac{1 -  n\gamma}{n\gamma} \cdot \frac{a}{\sqrt{1- t^2}}.
\]
Hence,
\[
\enorm{y - h} = \enorm{\frac{1}{n \gamma} \cdot \frac{q_1}{t}} =  \frac{1}{n \gamma} \cdot \frac{1}{t}
\]
and
\[
\enorm{h - x_{n+1}} = \frac{1}{\sqrt{1 - t^2}} \left(1 + \frac{1 - n\gamma}{n \gamma} \right) = \frac{1}{n \gamma} \cdot \frac{1}{\sqrt{1 - t^2}}.
\]

From the right triangle $x_{n+1} h y$,
\[
\enorm{y - x_{n+1}} = 
\sqrt{\enorm{y - h}^2 + \enorm{h - x_{n+1}}^2} = 
\frac{1}{n \gamma } \cdot \sqrt{ \frac{1}{1 - t^2} + \frac{1}{t^2} }.
\]

Since $yz$ is an angle bisector of the triangle $yhx_{n+1}$, we have:
\[
\frac{\enorm{h - z}}{\enorm{h - y}} = \frac{\enorm{z - x_{n+1}}}{\enorm{y - x_{n+1}}}.
\]

Therefore, the inradius of the simplex $x_1, \dots, x_{n+1}$ is
\[
R(t) = \enorm{h - z} = \frac{\enorm{y - h}}{\enorm{y - h} + \enorm{y - x_{n+1}}} \cdot \enorm{h - x_{n+1}}.
\]

Substituting the computed values, we obtain:
\[
R(t) = \frac{1}{n \gamma} \cdot \frac{1}{\sqrt{1 - t^2}} \cdot \frac{\frac{1}{t}}{\frac{1}{t} + 
\sqrt{ \frac{1}{1 - t^2} + \frac{1}{t^2} }} = 
\frac{1}{n \gamma} \cdot \frac{1}{\sqrt{1 - t^2} + \sqrt{t^2 + 1 - t^2}} =
\frac{1}{n \gamma} \cdot \frac{1}{1 + \sqrt{1 - t^2}}.
\]
\end{proof}

Denote $\sigma(t) =  \sum\limits_{i \in [n]} \gamma \enorm{v_i}^2 + (1 - n \gamma) \enorm{v_{n+1}}^2.$
Using the obtained expressions: 
  \[
  \sigma(t) = 
  n \gamma +  \frac{ \parenth{n\gamma}^2}{ 1 -  n\gamma}\parenth{1 - t^2} = 
  n \gamma \parenth{1  + \frac{n\gamma}{ 1 -  n\gamma}\parenth{1 - t^2}}.
\] 
\begin{thm}\label{thm:tightness_inrad_bound}
Fix non-negative integers $d,s$ and $m$ with $m > d-s\geq 2$ and a positive $\epsilon$. 
Then there is a configuration of $m$ vectors $b_1, \dots, b_{m}$ in $\Re^{d-s}$
 and positive numbers $\beta_1, \dots, \beta_{m}$
 satisfying 
 \[ \sum\limits_{i \in [m]} \beta_i = 1, 
\quad 
\sum\limits_{i \in [m]} \beta_i b_i =\origin , 
\quad
\sum\limits_{i \in [m]} \beta_i \enorm{b_i}^2 = \frac{d-s}{d},
\]
and such that the inradius of the polar of the convex hull of  
$\{b_1 \dots, b_m\}$ is at least 
 $ (1 - \epsilon)\frac{d}{d-s}.$
\end{thm}
\begin{proof}
Set  $\gamma = \frac{1}{d}$ and $n=d-s$.
Then, with the notation introduced previously, inequality \eqref{eq:example_gamma_condition} holds and 
the following  hold as $t$ tends to $1^{-}$:
   \begin{enumerate}
    \item the left-hand side of inequality \eqref{eq:example_last_v_ineq} tends to zero;
       \item $\sigma(t)$ is monotonically decreasing to $\frac{n}{d};$ 
\item  $
R(t) \to  \frac{d}{n}.
$
\end{enumerate}

The computations of limits are straightforward. 

For each $i \in [n],$ set $b_i = \sqrt{\frac{\theta}{\sigma(t)}} v_i$ with $\theta = \frac{d-s}{d}.$ 
Set $b_{n+1} = \ldots = b_m = \sqrt{\frac{\theta}{\sigma(t)}} v_{n+1}$. Also set $\beta_i=\gamma$ for $i\in[n]$ and $\beta_i=\frac{1-n\gamma}{m-n}$ for $i\in[m]\setminus[n]$.
Then clearly, $\sum\limits_{i\in[m]}\beta_i=1$. Moreover, 
\[
{
\sum_{i\in[m]}\beta_i b_i=\gamma \sum_{i \in [n]} b_i + \frac{(1 - n \gamma)}{m-n} \sum_{i \in [m] \setminus[n]} b_{i}
}=
\sqrt{\frac{\theta}{\sigma(t)}} 
\parenth{
\gamma \sum_{i \in [n]} v_i + (1 - n \gamma) v_{n+1}
} \stackrel{\eqref{eq:gammasum}}{=}\origin ;
\]
and
\[
\sum\limits_{i \in [m]} \beta_i \enorm{b_i}^2 =\gamma \sum_{i \in [n]} \enorm{b_i}^2 + \frac{(1 - n \gamma)}{m-n} \sum_{i \in [m] \setminus[n]} \enorm{b_{i}}^2 = \theta.
\]
 The inradius of the polar of $\conv\{b_1, \dots, b_{m}\}$
is greater than that of the simplex $v_1 \dots v_{d-s+1}$ for $t$ sufficiently close to $1.$ On the other hand, there is a $t$ such that the latter is at least $(1 - \epsilon)\frac{d}{d-s}.$
\end{proof}

\section{Majorization inequality for the volume bound}\label{sec:major_lemma}
The following lemma yields \eqref{eq:vr_product_bound} with the substitution $\enorm{P_F u_i}^2 = x_i$ and 
$d = d^\prime.$ We leave it to the reader that the constraints on $x_i$ are satisfied.
\begin{lem}
    \label{lem:vr_convex_func_max}
The  functional $W(x_1, \dots, x_m, \delta_1, \dots, \delta_m) = \prod\limits_{i \in [m]} x_i^{\frac{\delta_i x_i}{2}}$ 
under the constraints
\begin{itemize}
  \item $\frac{1}{d + 1} \leq x_i \leq 1 \quad \forall\, i \in [m],$
  \item $\sum\limits_{i \in [m]} \delta_i x_i = s + 1,$
  \item $\sum\limits_{i \in [m]} \delta_i = d + 1,$
  \item $0 \leq \delta_i \leq 1.$
\end{itemize}
on the variables satisfies the inequality
\[
W(x_1, \dots, x_m, \delta_1, \dots, \delta_m)  \leq \frac{(s+1)^{\frac{s+1}{2(d+1)}}}{(d+1)^{\frac{d+1}{2(d+1)}}}.
\]
\end{lem}
\begin{proof}
    Set $\Lambda = \frac{s+1}{d+1}.$ We will study the logarithm of 
    our functional. We have to show that
    \[
\ln W(x_1, \dots, x_m, \delta_1, \dots, \delta_m)  \leq \frac{1}{2} \Lambda \ln \Lambda - \frac{d-s}{2(d+1)} \ln(d+1).
\]
Note that the function $\ln W $ is continuous on the compact set in $\Re^{2m}$ defined by our constraints. Hence, it attains its maximum. 
For every $ x = (x_1, \ldots, x_m) $ with $ \frac{1}{d+1} \leq x_i \leq 1 $ for all $ i \in [m] $, let $ \ln W_x(\delta_1, \dots, \delta_m) $ be the function
\[
\ln W_x(\delta_1, \dots, \delta_m) =
\frac{1}{2} \sum\limits_{i \in [m]} \delta_i x_i \ln x_i.
\]

Notice that $\ln W_x$ is a convex function of
$\delta = (\delta_1, \dots, \delta_m) .$ Since the set
\[
A_x = \braces{\delta
\in \Re^m : \sum_{i \in [m]} \delta_i x_i = s+1, \ \sum_{i \in [m]} \delta_i = d+1, 
\ 0 \leq \delta_i \leq 1 \ \forall\, i \in [m] }
\]
is a compact convex set, $ \ln W_x $ attains its maximum on some extreme point of
$ A_x $.
These are the points of intersection of the 2-dimensional faces of the cube
\[
\braces{
\delta \in \Re^m \st 0 \leq \delta_i \leq 1 \ \forall\, i \in [m]
}
\]
with the $ (m - 2) $-dimensional affine subspace
\[
\braces{ \delta \in \Re^m : \sum\limits_{i \in [m]} \delta_i x_i = s + 1, \quad \sum\limits_{i \in [m]} \delta_i = d + 1}.
\]

Therefore, a maximizer of the function $ \ln W_x $ has to be a point of the form
\[
\delta_\lambda = (\delta_{\lambda,1}, \dots, \delta_{\lambda,m}) 
= \underbrace{(1, 1, \ldots, 1}_{d},\, \lambda,\, 1 - \lambda,\, \underbrace{0, \ldots, 0}_{m - d - 2})
\]
for some $ \frac{1}{2} \leq \lambda \leq 1 $ (or a permutation of it), such that $ \sum\limits_{i \in [m]} \delta_i x_i = s + 1 $ is satisfied.
For every $ \delta_\lambda $ with $ \frac{1}{2} \leq \lambda \leq 1 $, we will find the maximizer of the function
\[
\ln W_{\delta_\lambda}(x) = \frac{1}{2} \sum\limits_{i \in [m]} \delta_{\lambda, i} x_i \ln x_i
\]
on the compact convex set
\[
B_\lambda = \braces{ x \in \Re^m : \sum\limits_{i \in [m]} \delta_{\lambda, i} x_i = s + 1, \quad \frac{1}{d + 1} \leq x_i \leq 1 \quad \forall\, i \in [m] }.
\]

Set
\[
\tilde{x} = \parenth{
\underbrace{1, 1, \ldots, 1}_{s},\ 
\Lambda,\ 
\underbrace{\frac{1}{d + 1}, \ldots, \frac{1}{d + 1}}_{m - s - 1} 
}.
\]

We check that $ \frac{1}{d + 1} \leq \Lambda \leq 1 $ and
\[
\sum\limits_{i \in [m]} \delta_{\lambda, i} \tilde{x}_i 
= s + \Lambda + \frac{d - s}{d + 1} = s + 1.
\]
It is easy to see that $\tilde{x}$ majorizes any $x = (x_1, \dots, x_m) \in B_\lambda.$
Thus, by the weighted Karamata inequality (see \cite{fuchs1947new}), 
$\ln W_{\delta_\lambda} (x) \leq \ln W_{\delta_\lambda} \parenth{\tilde{x}}$ for all $x \in B_\lambda.$ By the definition of $\delta_\lambda,$ we conclude that the maximum under our constraints is equal to
\[
\max\limits_{\lambda \in [1/2,1]} \ln W_{\delta_\lambda} (x)  = 
\max\limits_{\lambda \in [1/2,1]} \ln W_{\delta_\lambda} =
\frac{1}{2} \Lambda \ln \Lambda - \frac{d-s}{2(d+1)} \ln(d+1).
\]
The lemma follows.
\end{proof}

\section*{Acknowledgements}
We thank Alexander Litvak and Roman Karasev for the fruitful conversations. We are also grateful to Susanna Dann, who organized the \emph{`Analysis and Convex Geometry Week at UniAndes'} in Bogotá in 2025, where the initial steps of this work were made. 

G. I. is supported by Projeto Paz and by CAPES (Coordena\c{c}\~ao de Aperfei\c{c}oamento de Pessoal de N\'ivel Superior) - Brasil, grant number 23038.015548/2016-06.
Z. L. and M. N. are partially supported by the NRDI research grant K147544.
M. N. is also supported by the NRDI research grant K131529 as well as the ELTE TKP 2021-NKTA-62 funding scheme. Z. L. is supported also by the ERC Advanced Grant ``ERMiD'',
Á. S. is funded by the grant NKFIH 150151, a project that has been implemented with the support provided by
the Ministry of Culture and Innovation of Hungary from the National Research, Development and Innovation Fund, financed under the ADVANCED\_24 funding scheme.

\bibliography{bibliography}
\bibliographystyle{amsalpha}
\end{document}